\documentclass[article, 12pt]{amsart}

\usepackage{amsmath, amssymb, amsfonts, amsthm, latexsym}
\usepackage{graphicx}
\usepackage{tikz}
\usepackage{enumitem}

\usetikzlibrary{arrows, decorations.markings}
\tikzstyle{vecArrow} = [thick, decoration={markings,mark=at position
   1 with {\arrow[semithick]{open triangle 60}}},
   double distance=1.4pt, shorten >= 5.5pt,
   preaction = {decorate},
   postaction = {draw,line width=1.4pt, white,shorten >= 4.5pt}]
\tikzstyle{innerWhite} = [semithick, white,line width=1.4pt, shorten >= 4.5pt]

\newtheorem{Theorem}{Theorem}[section]
\newtheorem{Lemma}[Theorem]{Lemma}
\newtheorem{Corollary}[Theorem]{Corollary}
\newtheorem{Proposition}[Theorem]{Proposition}

\newtheorem{Example}[Theorem]{Example}

\def\depth{\operatorname{depth}} 
 
\def\p{\operatorname{p}}
\def\core{\operatorname{core}}

\def\NN{{\mathbb N}}
\def\ZZ{{\mathbb Z}} 
\def\a{{\mathbf a}}
\def\b{{\mathbf b}}

\def\e{{\mathbf e}}

\def\1{{\mathbf 1}}
\def\mm{{\mathfrak m}}

\def\xx{{\rm \bf{x}}}

\def\F{{\mathcal F}}
\def\H{{\Omega}}

\def\D{{\Delta}}
\def\G{{\Gamma}}

\begin{document}
\title{Combinatorial characterizations\\ of the saturation and the associated primes\\
 of the fourth power of edge ideals}
\author{Ha Thi Thu Hien}
\address{Foreign Trade University, Hanoi, Vietnam}
\email{thuhienha504@gmail.com}

\author{Ha Minh Lam}
\address{Institute of Mathematics, Vietnam}
\email{hmlam@math.ac.vn}

\thanks{This research is funded by Vietnam National Foundation for Science and Technology Development (NAFOSTED) under the grant number 101.04-2014.52. Part of the paper was written while the second author visited VIASM, Vietnam Institute for Advanced Study in Mathematics. This author would like to thank VIASM for hospitality and support.}

\begin{abstract}
To compute the local cohomology of powers of edge ideals one needs to know their saturations.
The saturation of the second and third powers has been described in terms of the graph in \cite{TeraiTrung} and \cite{HLT}. In this article, we give a combinatorial description of the generators of the saturation of the fourth power. As a consequence, we are able to give a complete classification of the associated primes of the fourth power of edge ideals in terms of the graph.
\end{abstract}
\maketitle
\begin{center}
{\it Dedicate to Professor Ngo Viet Trung in honors of his sixtieth birthday}
\end{center}

\section*{Introduction}

Let $R = K[x_1, x_2, \ldots, x_n]$ be a polynomial ring over a field $K$ and $Q$ a monomial ideal in $R$. Then $R/Q$ is an $\NN^n$-graded algebra. Due to a formula of Y. Takayama \cite{Takayama}, the graded component of the local cohomology modules of degree $\a \in \ZZ^n$ can be expressed in terms of a simplicial complex $\D_\a$ on the vertex set $V = \{1,...,n\}$. By a recent work of  Terai and Trung \cite{TeraiTrung}, this simplicial complex can be described as follows.
\par

Let $\F(\D_a)$ denote the set of the factets of $\D_a$. If $\a = (a_1,...,a_n)$, we set $G_\a := \{i \in V|\ a_i < 0\}$. 
and $x^\a = x_1^{a_1}\cdots x_n^{a_n}$. For every subset $G \subseteq V$ let $\a_G$ denote the vector obtained from $\a$ by setting $a_i = 0$ for all $i \in G$. Then
$$\F(\D_\a)  = \{G \setminus G_\a|\ G_\a \subseteq G \subseteq V,\  x^{\a_G} \in \widetilde {Q_G} \setminus Q_G\},$$ 
where $Q_G$ is a well determined  ideal and $\widetilde {Q_G}$ denotes the saturation of $Q_G$. Thus, the simplicial complex $\D_a$ depends on the saturations of the ideals $Q_G$. \par

Let $\G$ be a simple graph on the vertex set $V$. We denote by $I(\G)$ the {\it{edge ideal}} of $\G$, which is generated by the monomials $x_ix_j , \{i, j\} \in \G$. For convenience set $I = I(\G)$. It has been shown in \cite{HLT} that $x^{\a_G} \in \widetilde {(I^t)_G} \setminus (I^t)_G$ if and only if $x^{\a_U} \in \widetilde {J^t} \setminus J^t$ for some subset $U \subseteq V$, where $J$ is the edge ideal of the reduced subgraph of $\G$ on $V \setminus U$. Therefore, to compute the local cohomology modules of $R/I^t$ we need to know the saturation of the $t$-th powers of edge ideals. \par

Moreover, if there exists a monomial $x^\a \in \widetilde {(I^t)_G} \setminus (I^t)_G$, then the ideal generated by the variables $x_i$,  $i \in V \setminus G,$ must be an associated prime of $I^t$, and every associated prime of $I^t$ arises in this way. Thus, if we can describe the monomials of the saturation of the $t$-th powers of edge ideals, we will be able to classify the associated primes of the $t$-th powers of edge ideals. 
Associated primes of powers of squarefree monomials ideals have been studied by many authors (see e.g. \cite{ChenMoreySung}, \cite{FHVt Critical}, \cite{FHVt odd}, \cite{FHVt Coloring}, \cite{HM}, \cite{MMV}, \cite{HS}). However, not much is known for the initial powers of $I^t$. \par

For $t = 2$, Terai and Trung \cite{TeraiTrung} showed that $x^\a \in \widetilde {I^2} \setminus I^2$ for some $\a \in \NN^n$ if and only if $x^\a$ is of the form $x_ix_jx_k$, where the induced subgraph on $\{i,j,k\}$ is a  dominating triangle of $\G$. Recall that a subset of $V$ is called dominating if every vertex of $\G$ is adjacent to this subset. From this it follows that every associated prime of $I^2$ corresponds to a cover of $\G$ which is minimal among the covers containing the closed neighborhood of a triangle of $\G$. In particular, $\depth R/I^2 = 0$ if and only if $\G$ has a dominating triangle. The latter result was discovered independently by Herzog and Hibi in \cite{HH}. \par

For $t \ge 3$, the problem is much more difficult because a monomial $x^\a \in \widetilde {I^t} \setminus I^t$ does not need not to be squarefree. By \cite{HLT}, one can associate with such a monomial a vertex weighted graph $\G_\a$ with certain good matching properties,  and one can describe all such weighted graphs for $t = 3$. From this it follows that there are five types of subgraphs such that every associated prime of $I^3$ corresponds to a cover of $\G$ which is minimal among the covers containing the closed neighborhood of such a subgraph of $\G$. It remains an open problem whether one can describe all monomials $x^\a \in \widetilde {I^t} \setminus I^t$ in combinatorial terms for $t \ge 4$. \par

 In this paper, we use the method of \cite{HLT} to solve the above problem for $t = 4$ and we give a complete classification of the associated primes of $I^4$. We give all types of subgraphs such that every associated prime of $I^4$ corresponds to a cover of $\G$ which is minimal among the covers containing the closed neighborhood of such a subgraph of $\G$. Together with the results of \cite{TeraiTrung}, \cite{HLT}, our results may be used to find combinatorial patterns of the associated primes of every power $I^t$ and to study the function $\depth R/I^t$.  \par

This paper is organized as follows. In Section 1, we discuss the relationship between weighted graphs and the saturation and the associated primes of $I^t$. In Section 2, we explicitly describe all weighted graphs associated to monomial in $\widetilde {I^4} \setminus I^4$ . Section 3 is devoted to the characterization of the saturation and the  associated primes of $I^4$. \par

We assume that the readers are familiar with basic notions in Commutative Algebra and Graph Theory, which can be found, for instance, in the books \cite{E} and \cite{We}.

\noindent{\it Acknowledgment}. The authors would like to thank Prof. Ngo Viet Trung for his guidance. 

\section{Weighted graphs and saturation}

A (vertex) {\em weighted graph} is a simple graph whose vertices are assigned a positive integers called {\em weight}. Unless otherwise specified, we denote by $a_i$ the weight of the vertex $i$. The simple graph alone is called the {\em base} of the weighted graph. Every simple graph is a weighted graph whose vertices have the {\em trivial weight}, i.e. weight one.
\par

Let $\H$ be a weighted graph on a vertex set $U$. A {\em matching} of $\H$ is a family of edges (not necessarily different) in which each vertex appears no more times than its weight. The largest possible number of edges of a matching of $\H$ is called the {\em matching number} of $\H$, denoted by $\nu(\H)$.  For every vertex $i \in U$, we denote by $N_\H(i)$ the set of all adjacent vertices of $i$ and by $\deg_\H(i)$ the sum of the weights of the vertices of $N_\H(i)$. For every subset $N$ of $U$ we denote by $\H - N$ the induced weighted subgraph obtained from $\H$ by deleting all vertices in $N$ and their adjacent edges. \par

Following \cite[Section 2]{HLT}, we call $\H$  a {\em $t$-saturating} weighted graph if the following conditions are satisfied: \par
{\rm (i) } $\nu(\H) < t$, \par
{\rm (ii)} $\nu(\H - N_\H(i)) \ge t - \deg_\H(i)$ for all $i \in U$. 

\begin{Example} \label{matching number} {\rm Let $\H$ be the weighted graph whose base is the union of a triangle $\{1,2,3\}$ and an edge $\{3,4\}$ and whose vertices have the weights $1,1,2,1$ (Figure 1). 

\begin{figure}[ht!]

\begin{tikzpicture}[scale=0.6]

\draw [thick] (8,0) coordinate (a) -- (8,2) coordinate (b) ;
\draw [thick] (8,2) coordinate (b) -- (9.5,1) coordinate (c) ;
\draw [thick] (9.5,1) coordinate (c) -- (8,0) coordinate (a) ; 
\draw [thick] (9.5,1) coordinate (c) -- (11.5,1) coordinate (d);

\draw (11.5,1) node[right] {$4$};
\draw (8,0) node[left, below] {$1$};
\draw (8,2) node[left, above] {$2$};
\draw (10,1) node[right,above] {$3 (2\times)$};    
\fill (a) circle (2pt);
\fill (b) circle (2pt);
\fill (c) circle (2pt);
\fill (d) circle (2pt);
  
\end{tikzpicture}
\caption{}
\end{figure}

It is easy to check that $\nu(\H) = 2$, $\nu(\H-N_\H(i)) = 2$ for $i = 1,2,4$, $\nu(\H-N_\H(3)) = 1$, and $\deg_\H(i) = 3$ for $i = 1,2,3$, $\deg_\H(4) = 2$. Hence, $\H$ is not $2$ saturating but $3$-saturating and $4$-saturating.}
\end{Example}

The notion of $t$-saturating weighted graphs was introduced in order to study the saturation of the $t$-th power of edge ideals. \par

Let $\G$ be a simple graph on $V = \{1,...,n\}$. As before, we denote by $I = I(G)$ the edge ideal of $\G$. Given a vector $\a \in \NN^n$, we denote by $\G_\a$ the weighted graph whose base is the induced subgraph of $\G$ on the vertex set $V_\a:= \{i |\ a_i > 0\}$ and whose vertices $i$ has the weights $a_i$ for each $i \in V_\a$. For a vertex $i \in V$, we denote by $N_\a(i)$ the set of all adjacent vertices of $i$ in $\G_\a$ and set $\deg_{\a}(i) = \sum_{j \in N_\a(i)} a_j$. Note that $N_{\G_\a}(i) = N_\a(i)$ and $\deg_{\G_\a}(i) = \deg_\a(i)$ for $i \in V_\a$.

\begin{Theorem}\label{satthen} \cite[Theorem 2.1]{HLT}
$\xx^\a \in  \tilde{I^t} \setminus I^t$ if and only if the following conditions are satisfied:\par
{\rm (i)} $\G_\a$ is $t$-saturating, \par
{\rm (ii)}  $\nu(\G_\a  - N_\a(i)) \ge t - \deg_\a(i)$ for all $i \in V \setminus V_\a$. 
\end{Theorem}

In many cases, condition (ii) of Theorem \ref{satthen} is equivalent to the following property of $V_\a$.

\begin{Corollary} \label{dominating} \cite[Corollary 2.2]{HLT}
If $x^\a \in  \widetilde{I^t} \setminus I^t$, then $V_\a$ is a dominating set of $\G$.
\end{Corollary}

Using the notion of $t$-saturating graphs one can also characterize the associated primes of $I^t$ as follows. \par

For a subset $F \subseteq V$ we denote by $P_F$ the ideal generated by the variables $x_i$, $i \in F.$ It is well known that the associated primes of $I^t$ are of the form $P_F$, where $F$ is a cover of $\G$, i.e a set of vertices which meets every edge of $\G$. It is also known that $P_F$ is a minimal associated primes of $I^t$ if and only if $F$ is a minimal cover of $\G$. Thus, we only need to characterize the embedded (i.e. non-minimal) associated primes of $I^t$. \par

Let $\core(F)$ denote the set of vertices in $F$ which are not adjacent to any vertex in $V \setminus F$. Note that $\core(F) = V \setminus N[V \setminus F]$, where for a subset $U  \subseteq V$ we denote by $N[U]$ the closed neighborhood of $U$, i.e. the union of $U$ with the set of the vertices adjacent to some vertex of $U$. It is easy to see that a cover $F$ is minimal if and only if $\core(F) = \emptyset$. 

\begin{Theorem} \label{associated prime It} \cite[Theorem 2.4]{HLT} 
Let $F$ be a cover of $\G$ with $\rm{core}(F) \neq \emptyset$. Then $P_F$ is an embedded associated prime of $I^t$ if and only if $F$ is minimal among the covers containing $N[V_\a]$ for some $\a \in \NN^n$ such that  $\G_\a$ is $t$-saturating and $\nu(\G_\a  - N_\a(i)) \ge t - \deg_\a(i)$ for all $i \in \rm{core}(F) \setminus V_\a$.  
\end{Theorem}
 
According to the above two theorems, to compute the saturation and the associated primes of $I^t$ one needs to know how a $t$-saturating weighted graph looks like. \par

In general,  one can reduce the description of a $t$-saturating weighted graphs to the case of a simple graph
as follows.   \par

Let $\H$ be a weighted graph. The {\emph{polarization}} $\p(\H)$ of $\H$ is the simple graph obtained from the base of $\H$ by replacing each vertex $i$ by $a_i$ new vertices $i_1,... ,i_{a_i}$ and every edge $\{i,j\}$ by $a_ia_j$ edges $\{i_t,j_u\}$, $t = 1,...,a_i$ and $u = 1,...,a_j$.  By this definition, all vertices $i_1,... ,i_{a_i}$ of $p(\H)$ are non-adjacent and have the same neighborhood. 

\begin{Example} \label{pol of a graph} 
{\rm Let $\H$ be the weighted graph of Example \ref{matching number}. Then $\p(\H)$ can be visualized as in Figure 2.}
\end{Example}

\begin{figure}[ht!]

\begin{tikzpicture}[scale=0.6]

\draw [thick] (8,0) coordinate (a) -- (8,2) coordinate (b) ;
\draw [thick] (8,2) coordinate (b) -- (9.5,1) coordinate (c) ;
\draw [thick] (9.5,1) coordinate (c) -- (8,0) coordinate (a) ; 
\draw [thick] (9.5,1) coordinate (c) -- (11.5,1) coordinate (d);

\draw (10.25,-1.30) node[above] {$\H$};

\draw (11.5,1) node[right] {$4$};
\draw (8,0) node[left, below] {$1$};
\draw (8,2) node[left, above] {$2$};
\draw (10,1) node[right,above] {$3 (2\times)$};    
\fill (a) circle (2pt);
\fill (b) circle (2pt);
\fill (c) circle (2pt);
\fill (d) circle (2pt);

\draw (13,1) node[right] {$\longrightarrow$};
  
\draw [thick] (16,0) coordinate (a) -- (16,2) coordinate (b) ;
\draw [thick] (16,2) coordinate (b) -- (17.5,1) coordinate (c) ;
\draw [thick] (17.5,1) coordinate (c) -- (16,0) coordinate (a) ; 
\draw [thick] (17.5,1) coordinate (c) -- (19.5,1) coordinate (d);
\draw [thick] (16,2) coordinate (b) -- (17.5,3) coordinate (e) ;
\draw [thick] (16,0) coordinate (a) -- (17.5,3) coordinate (e) ;
\draw [thick] (17.5,3) coordinate (e) -- (19.5,1) coordinate (d);
 
\draw (18.20,-1.30) node[above] {$\p(\H)$};

\draw (19.5,1) node[right] {$4$};
\draw (16,0) node[left, below] {$1$};
\draw (16,2) node[left, above] {$2$};
\draw (17.5,1) node[right,above] {$3_1$}; 
\draw (17.5,3) node[right,above] {$3_2$};

\fill (a) circle (2pt);
\fill (b) circle (2pt);
\fill (c) circle (2pt);
\fill (d) circle (2pt);
 
\end{tikzpicture}
\caption{}
\end{figure}
 
\begin{Lemma}   \label{matching}
$\nu(\H) = \nu(\p(\H))$.
\end{Lemma} 

\begin{proof} Consider the map which sends every edge $\{u_i,v_j\}$ of $\p(\H)$ to the edge $\{u,v\}$ of $\H$. By this map, the family of the edges of $\H$ which are the images of the edges of a matching of $\p(\H)$ is a matching of $\H$ of the same cardinality. Conversely, for every matching $M$ of $\H$ we can find a matching $M^*$ of $\p(\H)$ of the same cardinality such that every edge of $M$ is the image of an edge of  $M^*$. This follows from the fact that for every vertex $i$, $a_i$ is the number of vertices of $\p(\H)$ arising from $i$ and that $a_i$ is greater than or equal to the number of appearing times of the vertex $i$ in $M$. Therefore, $\nu(\H) = \nu(\p(\H))$.
\end{proof}

\begin{Proposition}  \label{polar}
$\H$ is $t$-saturating if and only if so is $\p(\H)$.
\end{Proposition}

\begin{proof}
For every vertex $i$ of $\H$, one can easily check that  $\p(\H) - N_{\p(H)}(i_t)  = \p(\H - N_\H(i))$ and $\deg_{\p(H)}(i_t) = \deg_\H(i)$ for all $t = 1,...,a_i$.  Therefore, the conclusion follows from the definition of $t$-saturating graphs and Lemma \ref{matching}.
\end{proof}

Conversely, one can obtain from a $t$-saturating simple graph a $t$-saturating weighted graph by collapsing non-adjacent vertices with the same neighborhood. \par

Let $v_1$,...,$v_r$ are non-adjacent vertices of a weighted graph $\H$ which have the same neighborhood. Let $\H'$ be a weighted graph  obtained from $\H$ by replacing $v_1$,...,$v_r$ by a single new vertex $v$ with weight $a_v := \sum_{i =1}^r a_{v_i}$ and the edges of the form $\{v_i,w\}$ by a single edge $\{v,w\}$. We call $\H'$ a {\emph{collapsed weighted graph}} of $\H$. Obviously, every weighted graph $\H$ is a collapsed weighted graph of its polarization $\p(\H)$.  

\begin{Lemma}   \label{matching collap}
Let $\H'$ be a collapsed weighted graph of $\H$. Then $\nu(\H') = \nu(\H)$. 
\end{Lemma}

\begin{proof}
We consider the map which sends every edge of the form $\{v_i,w\}$ of $\H$, $i = 1,...,r$, to the edge $\{v,w\}$ of $\H'$ and other edges of $\H$ to the same edge of $\H'$. Using this map we obtain the assertion as in the proof of Lemma \ref{matching}.  
\end{proof}
 
\begin{Proposition}  
Let $\H'$ be a collapsed weighted graph of $\H$. Then $\H$ is $t$-saturating if and only if so is $\H'$.
\end{Proposition}

\begin{proof}
Similar to the proof of Proposition \ref{polar}, the conclusion follows from the definition of collapsed weighted graphs and Lemma \ref{matching collap}.
\end{proof}

By definition, a $t$-saturating weighted graph $\H$ is $(t-1)$-saturating if $\nu(\H) < t-1$.
Therefore, we may concentrate our investigation on $t$-saturating weighted graphs with matching number $t-1$.
We will present below a property of $t$-saturating simple graphs with matching number $t-1$, which will be useful in our investigation. \par

Let $\G$ be a simple graph on the vertex set $V = \{1,...,n\}$. Let $M$ be an arbitrary matching of $\G$.
A path or a cycle $P$ in $\G$ is called $M$-augmenting if $P$ starts and ends on vertices not in $M$, and alternates between edges in and not in $M$. Note that for a $M$-augmenting cycle, there is only an endpoint and there
are two adjacent edges not in $M$ meeting at the endpoint. It is clear that an augmenting path is always odd. Recall that $M$ is a maximum matching if $|M| = \nu(\G)$. \par
 
\begin{Lemma} \label{berge} 
Let $\G$ be a $t$-saturating graph with $\nu(\G)=t-1$ and $M$ a maximum matching of $\G$. 
For every vertex $i$ not in $M$, there is an $M$-augmenting cycle $P_i$ which starts and ends on $i$. 
Moreover, for every pair of vertices $j \neq i$ in $M$, two such cycles $P_i$ and $P_j$ are disjoint.
\end{Lemma}

\begin{proof} First, we construct a new simple graph $\G_i^\prime$ from $\G$ by replacing the vertex $i$ by $d:=\deg_\G(i)$ new vertices $i_1, i_2, ...,i_d$, and replacing every edges $\{i,j\}$ by $d$ edges $\{i_1,j\}, \{i_2,j\}, ...,\{i_d,j\}.$ The induced subgraph of $\G_i^\prime$ on the set $N_\G(i) \cup \{i_1,...,i_d\}$ contains $d$ disjoint edges. Adding these $d$ edges to any matching of $\G_i^\prime - N_{\G_i^\prime}(i)$ we obtain a matching of $\G_i^\prime$. Therefore, $\nu(\G_i^\prime) \ge \nu(\G' - N_\G(i)) + d.$ 
Since $\G-N_\G(i) = \G_i^\prime - N_\G(i)$, this implies $\nu(\G_i^\prime) \ge \nu(\G - N_\G(i)) + \deg_\G(i) \ge t$.\par 

By identifying $i$ with $i_1$ we can consider $\G$ as a subgraph of $\G_i^\prime$ and $M$ as a matching of $\G_i^\prime.$ Since $M$ contains $t-1$ edges while $\nu(\G_i^\prime) \ge t$, $M$ is not a maximum matching of $\G_i^\prime$. Due to Berge's Lemma \cite{Be}, there is an $M$-augmenting-path $P$ in $\G_i^\prime$. The edges of $P$ not in $M$ form a matching $M'$ of $\G_i^\prime$ with more edges than $M$. Therefore, $M'$ is not a matching of $\G$. From this it follows that the end vertices of $P$ must belong to $\{i_1,...,i_d\}$. Since $i$ is not in $M$, the other vertices of $P$ do not belong to  $\{i_1,...,i_d\}$. Thus, if we replace $i_1,...,i_s$ by $i$ we obtain from $P$ an $M$-augmenting cycle $P_i$ in $\G$ which starts and ends at $i$. \par 

Assume that there exist two different vertices $i$ and $j$ not in $M$ such that there are two such $M$-augmenting cycles $P_i$ and $P_j$ which are not disjoint. Let $v$ be the first vertex of $P_i$ (as a closed path starting from $i$) which belongs to $P_j$. Then $v$ is in $M$. If the path from $i$ to $v$ on $P_i$ ends with an edge in (or not in)  $M$, we can find a path from $v$ to $j$ on $P_j$  which starts with an edge not in (or in) $M$. By connecting the two paths we obtain an $M$-augmenting path from $i$ to $j$. By Berge's Lemma, this is a contradiction to the assumption that $M$ is a maximum matching. Therefore, we can conclude that $P_i$ and $P_j$ are always disjoint if $i \neq j$.
\end{proof}


\section{Classification of $4$-saturating weighted graphs}

In this section, we will describe all $4$-saturating weighted graphs.  By definition, a weighted graph $\H$ on the vertex set $V$ is $4$-saturating if the following conditions are satisfied:\par 
{\rm (i) } $\nu(\G) < 4$, \par
{\rm (ii)} $\nu(\G - N_\G(i)) \ge 4 - \deg_\G(i)$ for all $i \in V$. \par
If, in addition, $\nu(\G) < 3$ then $\G$ is also 3-saturating. The 3-saturating weighted graphs have been classified in \cite{HLT}. 

\begin{center}
\begin{picture}(0,0)%
\includegraphics{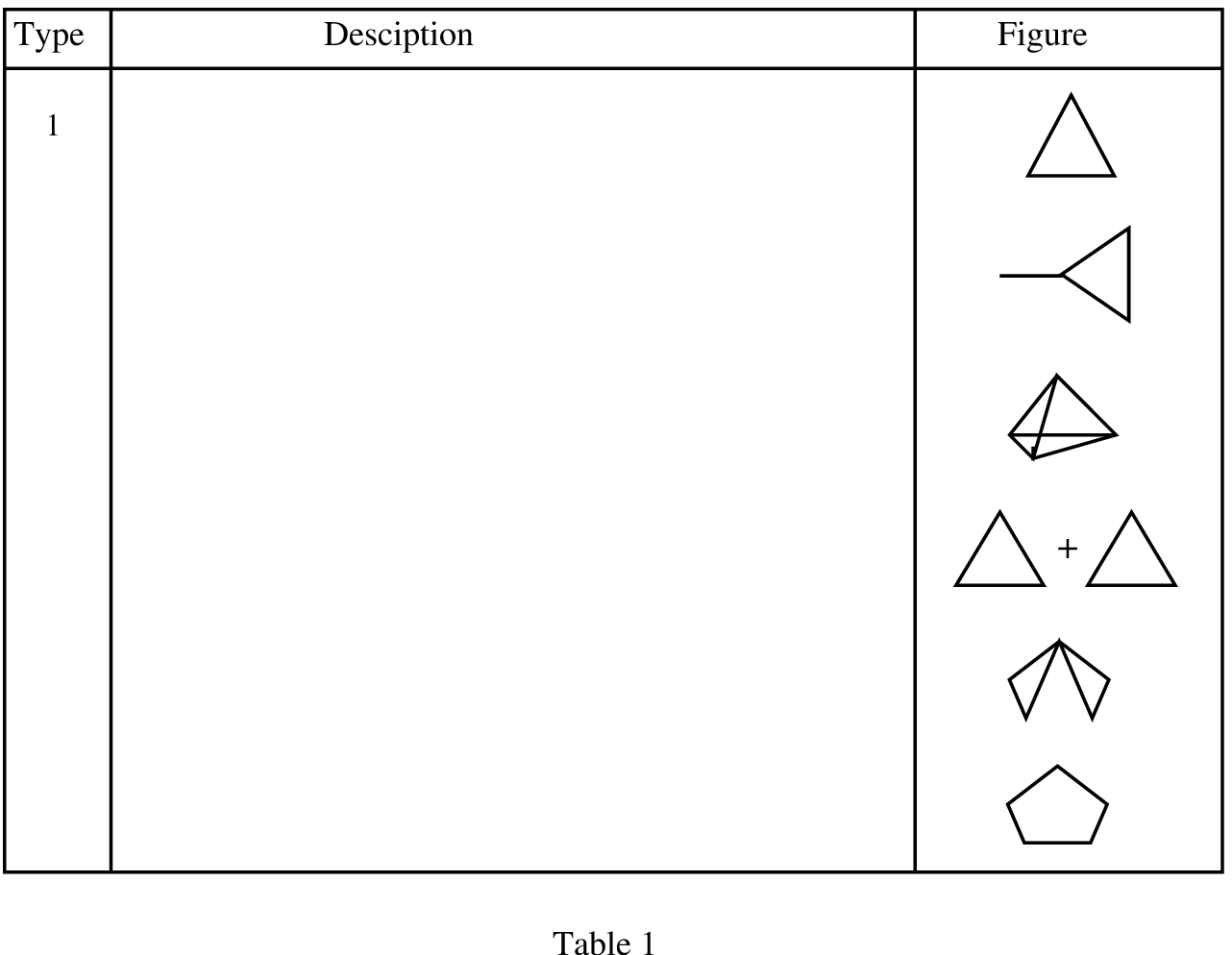}%
\end{picture}%
\setlength{\unitlength}{4144sp}%
\begingroup\makeatletter\ifx\SetFigFont\undefined
\def\x#1#2#3#4#5#6#7\relax{\def\x{#1#2#3#4#5#6}}%
\expandafter\x\fmtname xxxxxx\relax \def\y{splain}%
\ifx\x\y   
\gdef\SetFigFont#1#2#3{%
  \ifnum #1<17\tiny\else \ifnum #1<20\small\else
  \ifnum #1<24\normalsize\else \ifnum #1<29\large\else
  \ifnum #1<34\Large\else \ifnum #1<41\LARGE\else
     \huge\fi\fi\fi\fi\fi\fi
  \csname #3\endcsname}%
\else
\gdef\SetFigFont#1#2#3{\begingroup
  \count@#1\relax \ifnum 25<\count@\count@25\fi
  \def\x{\endgroup\@setsize\SetFigFont{#2pt}}%
  \expandafter\x
    \csname \romannumeral\the\count@ pt\expandafter\endcsname
    \csname @\romannumeral\the\count@ pt\endcsname
  \csname #3\endcsname}%
\fi
\fi\endgroup
\begin{picture}(5838,4607)(372,-4186)
\put(1126,-3436){\makebox(0,0)[lb]{\smash{\SetFigFont{10}{12.0}{rm}$\Omega$ is spanned by a pentagon}}}
\put(1126,-2836){\makebox(0,0)[lb]{\smash{\SetFigFont{10}{12.0}{rm}$\Omega$ is spanned by two triangles meeting at a vertex}}}
\put(588,-2256){\makebox(0,0)[lb]{\smash{\SetFigFont{10}{12.0}{rm}4}}}
\put(1126,-2256){\makebox(0,0)[lb]{\smash{\SetFigFont{10}{12.0}{rm}$\Omega$ is a disjoint union of 2 triangles}}}
\put(588,-1636){\makebox(0,0)[lb]{\smash{\SetFigFont{10}{12.0}{rm}3}}}
\put(588,-986){\makebox(0,0)[lb]{\smash{\SetFigFont{10}{12.0}{rm}2}}}
\put(1126,-1636){\makebox(0,0)[lb]{\smash{\SetFigFont{10}{12.0}{rm}$\Omega$ is a $K_4$}}}
\put(1126,-1129){\makebox(0,0)[lb]{\smash{\SetFigFont{10}{12.0}{rm}a vertex of weight 2}}}
\put(1126,-904){\makebox(0,0)[lb]{\smash{\SetFigFont{10}{12.0}{rm}$\Omega$ is spanned by a triangle and an edge meeting at}}}
\put(1126,-229){\makebox(0,0)[lb]{\smash{\SetFigFont{10}{12.0}{rm}$\Omega$ is a triangle with weight $(2,2,1)$}}}
\put(5401,-1113){\makebox(0,0)[lb]{\smash{\SetFigFont{10}{12.0}{rm}$2$}}}
\put(5151,-455){\makebox(0,0)[lb]{\smash{\SetFigFont{10}{12.0}{rm}$2$}}}
\put(5696,-455){\makebox(0,0)[lb]{\smash{\SetFigFont{10}{12.0}{rm}$2$}}}
\put(588,-2836){\makebox(0,0)[lb]{\smash{\SetFigFont{10}{12.0}{rm}5}}}
\put(588,-3486){\makebox(0,0)[lb]{\smash{\SetFigFont{10}{12.0}{rm}6}}}
\end{picture}

\end{center}

\begin{Theorem} \label{3-saturating} \cite[Theorem 4.3]{HLT}
$\H$ is 3-saturating if and only if $\H$ belongs to the types of Table 1.
\end{Theorem}

From this classification we can easily see which 3-saturating weighted graphs are 4-saturating. 

\begin{Corollary} \label{nu < 3} 
Let $\H$ be a 4-saturating weighted graph with $\nu(\H) < 3$.  Then $\H$ is the complete graph $K_5$.
\end{Corollary}

\begin{proof} 
It is easy to check that $\H$ is not of Types 1-4 of Table 1 because the above condition (ii) is not satisfied. Thus, $\H$ is a simple graph spanned by a pentagon, or by two triangles meeting at a vertex. For such a graph, condition (ii) is satisfied if and if $\deg_\H(i) = 5$ for all vertices $i$. This means that $\H$ is the complete graph $K_5$. 
\end{proof}

Now, we will classify all 4-saturating simple graphs. 

\begin{figure}
\begin{center}
\begin{picture}(0,0)%
\includegraphics{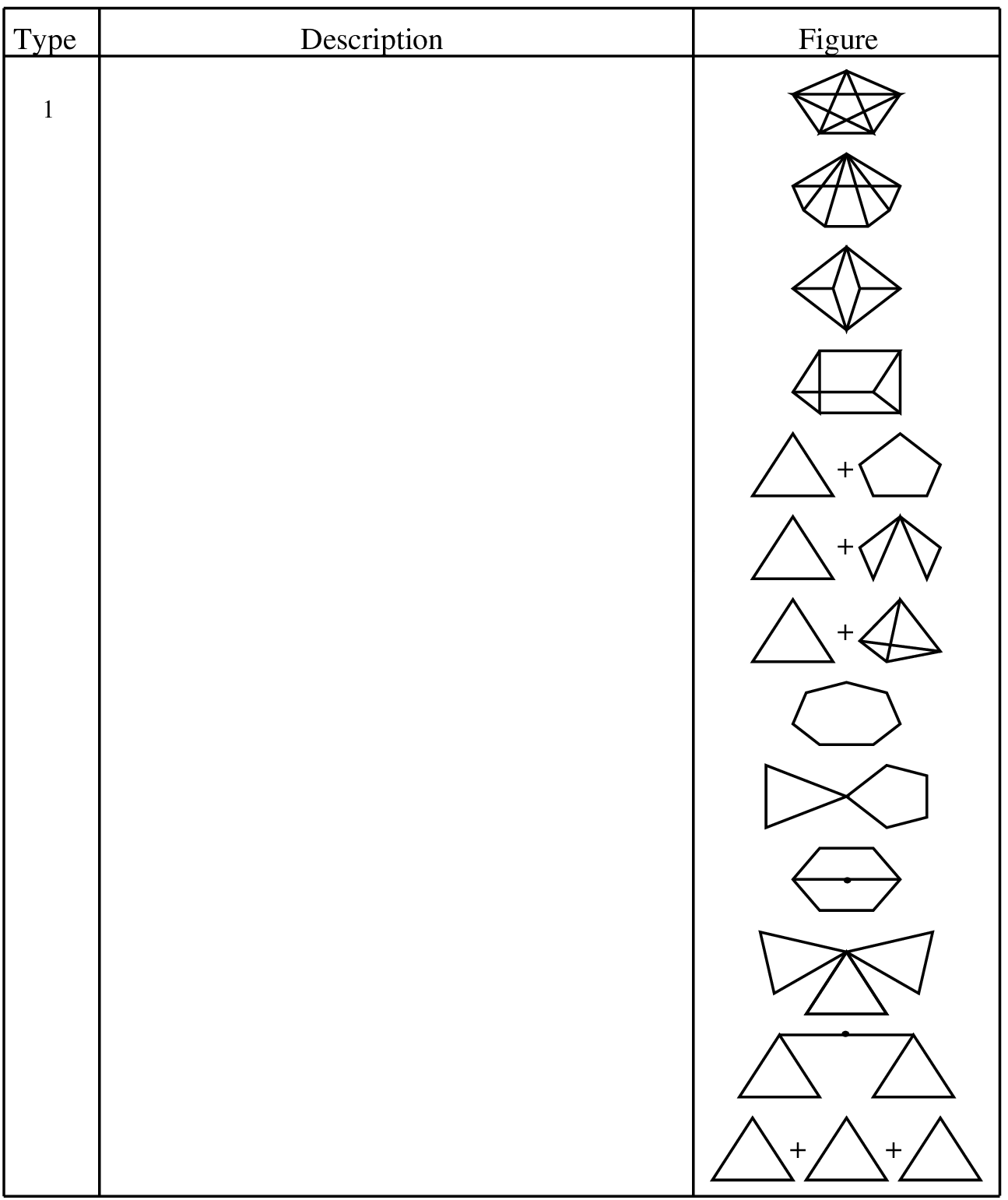}%
\end{picture}%
\setlength{\unitlength}{4144sp}%
\begingroup\makeatletter\ifx\SetFigFont\undefined
\def\x#1#2#3#4#5#6#7\relax{\def\x{#1#2#3#4#5#6}}%
\expandafter\x\fmtname xxxxxx\relax \def\y{splain}%
\ifx\x\y   
\gdef\SetFigFont#1#2#3{%
  \ifnum #1<17\tiny\else \ifnum #1<20\small\else
  \ifnum #1<24\normalsize\else \ifnum #1<29\large\else
  \ifnum #1<34\Large\else \ifnum #1<41\LARGE\else
     \huge\fi\fi\fi\fi\fi\fi
  \csname #3\endcsname}%
\else
\gdef\SetFigFont#1#2#3{\begingroup
  \count@#1\relax \ifnum 25<\count@\count@25\fi
  \def\x{\endgroup\@setsize\SetFigFont{#2pt}}%
  \expandafter\x
    \csname \romannumeral\the\count@ pt\expandafter\endcsname
    \csname @\romannumeral\the\count@ pt\endcsname
  \csname #3\endcsname}%
\fi
\fi\endgroup
\begin{picture}(5894,7019)(654,-6608)
\put(3250,-6800){\makebox(0,0)[lb]{\smash{\SetFigFont{10}{12.0}{rm}Table 2}}}
\put(1350,-6307){\makebox(0,0)[lb]{\smash{\SetFigFont{10}{12.0}{rm}$\Omega$ is a disjoint union of 3 triangles}}}

\put(1350,-5940){\makebox(0,0)[lb]{\smash{\SetFigFont{10}{12.0}{rm}a path of length 2}}}
\put(1350,-5757){\makebox(0,0)[lb]{\smash{\SetFigFont{10}{12.0}{rm}$\Omega$ is spanned by 2 triangles connected by }}}

\put(1350,-5390){\makebox(0,0)[lb]{\smash{\SetFigFont{10}{12.0}{rm}at a vertex}}}
\put(1350,-5150){\makebox(0,0)[lb]{\smash{\SetFigFont{10}{12.0}{rm}$\Omega$ is spanned by a union of 3 triangles meeting }}}

\put(1350,-4840){\makebox(0,0)[lb]{\smash{\SetFigFont{10}{12.0}{rm}sharing a path of length 2}}}
\put(1350,-4656){\makebox(0,0)[lb]{\smash{\SetFigFont{10}{12.0}{rm}$\Omega$ is spanned by a union of 2 pentagons }}}

\put(1350,-4381){\makebox(0,0)[lb]{\smash{\SetFigFont{10}{12.0}{rm}pentagon meeting at a vertex}}}
\put(1350,-4198){\makebox(0,0)[lb]{\smash{\SetFigFont{10}{12.0}{rm}$\Omega$ is spanned by a union of a triangle and a}}}

\put(1350,-3774){\makebox(0,0)[lb]{\smash{\SetFigFont{10}{12.0}{rm}$\Omega$ is spanned by a $C_7$}}}

\put(1350,-3316){\makebox(0,0)[lb]{\smash{\SetFigFont{10}{12.0}{rm}$\Omega$ is a disjoint union of a triangle and a $K_4$}}}

\put(1350,-2970){\makebox(0,0)[lb]{\smash{\SetFigFont{10}{12.0}{rm}spanned by 2 triangles meeting at a vertex}}}
\put(1350,-2730){\makebox(0,0)[lb]{\smash{\SetFigFont{10}{12.0}{rm}$\Omega$ is a disjoint union of a triangle and a graph}}}

\put(1406,-2419){\makebox(0,0)[lb]{\smash{\SetFigFont{10}{12.0}{rm}spanned by a $C_5$}}}
\put(1350,-2180){\makebox(0,0)[lb]{\smash{\SetFigFont{10}{12.0}{rm}$\Omega$ is a disjoint union of a triangle and a graph}}}

\put(1350,-1813){\makebox(0,0)[lb]{\smash{\SetFigFont{10}{12.0}{rm}$\Omega$ is spanned by a triangular prism}}}

\put(1350,-1263){\makebox(0,0)[lb]{\smash{\SetFigFont{10}{12.0}{rm}$\Omega$ is spanned by a union of 4 triangles}}}

\put(1350,-748){\makebox(0,0)[lb]{\smash{\SetFigFont{10}{12.0}{rm}$\Omega$ is spanned by a cone on $C_5$}}}

\put(1350,-197){\makebox(0,0)[lb]{\smash{\SetFigFont{10}{12.0}{rm}$\Omega$ is a $K_5$}}}

\put(957,-6307){\makebox(0,0)[lb]{\smash{\SetFigFont{10}{12.0}{rm}13}}}
\put(900,-5757){\makebox(0,0)[lb]{\smash{\SetFigFont{10}{12.0}{rm}12}}}

\put(956,-5242){\makebox(0,0)[lb]{\smash{\SetFigFont{10}{12.0}{rm}11}}}
\put(900,-4691){\makebox(0,0)[lb]{\smash{\SetFigFont{10}{12.0}{rm}10}}}

\put(900,-4290){\makebox(0,0)[lb]{\smash{\SetFigFont{10}{12.0}{rm}9}}}
\put(900,-3774){\makebox(0,0)[lb]{\smash{\SetFigFont{10}{12.0}{rm}8}}}
\put(900,-3316){\makebox(0,0)[lb]{\smash{\SetFigFont{10}{12.0}{rm}7}}}
\put(900,-2822){\makebox(0,0)[lb]{\smash{\SetFigFont{10}{12.0}{rm}6}}}
\put(900,-2272){\makebox(0,0)[lb]{\smash{\SetFigFont{10}{12.0}{rm}5}}}
\put(900,-1869){\makebox(0,0)[lb]{\smash{\SetFigFont{10}{12.0}{rm}4}}}
\put(900,-1319){\makebox(0,0)[lb]{\smash{\SetFigFont{10}{12.0}{rm}3}}}
\put(900,-691){\makebox(0,0)[lb]{\smash{\SetFigFont{10}{12.0}{rm}2}}}
\end{picture}

\end{center}
\end{figure}
$\  $\par
$ $
\begin{Theorem} \label{squarefree} 
$\H$ is 4-saturating if and only if $\H$ belongs to the types of Table 2.
\end{Theorem}

\begin{proof} 
It is easy to verify that all graphs in Table 2 are $4$-saturating. It remains to show that if $\H$ is 
a $4$-saturating graph, then $\H$ belongs to the types of Table 2. Note that from condition (ii) it follows that for every $i \in \H$, $\deg_\H(i) \ge 4 - \nu(\H - N_\H(i)) > 4- \nu(\H).$ Combining with condition (i), we get $\deg_\H(i) \ge 2 \ \forall i$, or, equivalently, every vertex of $\H$ has at least 2 neighbors. \par

If $\nu(\H)<3$, then $\H$ is of Type 1 by Lemma \ref{nu < 3}. Now assume that $\nu(\H)=3.$ Then $\H$ has three disjoint edges. Hence $|V| = n \ge 6$. On the other hand, $n \le 9$ by \cite[Lemma 2.5]{HLT}. \par 

\noindent {\underline{Case 1}}: $n = 6$. Since every vertex $i$ has at least 2 neighbors, $\H-N_\H(i)$ has at most 3 vertices, and hence $\nu(\H-N_\H(i))\le 1$ $\forall i \in \H$. From condition (ii) it follows that $5 \ge \deg_\H(i) \ge 3$ for all $i \in \H$. We distinguish the following subcases: \par

\noindent {\underline{Subcase 1.1}}: $\max \deg_\H(i) = 5$. Let $i$ be a vertex with $\deg_\H(i) = 5$. Since $\nu(\H)=3$, one has $\nu(\H-i) = 2$. Then $\nu(\H-i) = \nu(\H) - 1.$ Moreover, one has
$$(\H-i) -N_{\H-i}(j) = \H - N_\H(j),\\ \mbox{ for all } j\neq i,$$
$$\deg_{\H-i}(j) = \deg_\H(j) -1, \\ \mbox{ for all } j\neq i.$$
Hence, $\H$ is 4-saturating if and only if the induced graph on $V(\H) -\{i\}$ is 3-saturating. It follows from Theorem \ref{3-saturating} that $\H - i$ is either a pentagon, or the union of two triangles meeting at a vertex. Then $\H$ is either of Type 2, or of Type 3. \par

\begin{figure}[ht!]

\begin{tikzpicture}[scale=0.6]






\draw [thick] (9,6) coordinate (a) -- (7.5,4) coordinate (b) ;
\draw [thick] (a) -- (8.5,3.5) coordinate (c) ;
\draw [thick] (a) -- (9.5,3.5) coordinate (d) ;
\draw [thick] (a) -- (10.5,4) coordinate (e) ;
\draw [thick] (9,2) coordinate (x) -- (b) ;
\draw [thick] (x) -- (c) ;
\draw [thick] (x) -- (d) ;
\draw [thick] (x) -- (e) ;
\draw [thick] (b) -- (c) ;
\draw [thick] (d) -- (e) ;

\fill (a) circle (2pt) ;
\fill (b) circle (2pt) ;
\fill (c) circle (2pt) ;
\fill (d) circle (2pt) ;
\fill (e) circle (2pt) ;
\fill (x) circle (2pt) ;

\draw (a) node[above] {$i$} ;
\draw (b) node[left] {$p$} ;
\draw (8.3,3.5) node[left,above] {$q$} ;
\draw (9.7,3.5) node[right,above] {$u$} ;
\draw (e) node[right] {$v$} ;
\draw (x) node[below] {$j$} ;
\draw (9,0.7) node[below] {$\{u,v\} \in \H$} ;


\draw [thin] (18,6) coordinate (u) -- (16.5,4) coordinate (o) ;
\draw [very thick] (u) -- (17.5,3.5) coordinate (p) ;
\draw [very thick] (u) -- (18.5,3.5) coordinate (q) ;
\draw [very thick] (u) -- (19.5,4) coordinate (r) ;
\draw [very thick] (18,2) coordinate (v) -- (o) ;
\draw [thin] (v) -- (p) ;
\draw [very thick] (v) -- (q) ;
\draw [very thick] (v) -- (r) ;
\draw [very thick] (o) -- (p) ;
\draw [very thick] (o) -- (q) ;
\draw [very thick] (p) -- (r) ;

\fill (u) circle (2pt) ;
\fill (v) circle (2pt) ;
\fill (o) circle (2pt) ;
\fill (p) circle (2pt) ;
\fill (q) circle (2pt) ;
\fill (r) circle (2pt) ;

\draw (u) node[above] {$i$} ;
\draw (o) node[left] {$p$} ;
\draw (17.3,3.7) node[left,above] {$q$} ;
\draw (18.7,3.7) node[right,above] {$u$} ;
\draw (r) node[right] {$v$} ;
\draw (v) node[below] {$j$} ;
\draw (18,0.7) node[below] {$\{u,v\} \notin \H$} ;
\end{tikzpicture}
\caption{ }
\end{figure}

\noindent \underline{Subcase 1.2}: $\max \deg_\H(i) = 4$. Let $i$ be a vertex with $\deg_\H(i) = 4$. Let $p,q,u,v$ be the neighbors of $i$ and let $j$ be the vertex non-adjacent to $i$. 

If $\deg_\H(j) = 4$, then one has $N_\H(j) = \{p,q,u,v\}$. Since $\deg_\H(p) \ge 3$ then $p$ is adjacent to either $q$, or $u$, or $v$. Without loss of generality, we can assume that $\{p, q\}$ is an edge of $\H$ (see Figure 3). If $\{u,v\}$ is an edge of $\H$ then $\H$ is of type 3. If $\{u,v\} \notin \H$ then $\nu(\H-N_\H(p)) \le \nu(\H|_{\{u,v\}}) = 0$. Combining with the condition $4 - \deg_\H(p) \le\nu(\H-N_\H(p))$ one gets $4 - \deg_\H(p) = 0$. Hence, either $\{p,u\}$ or $\{p,q\}$ is an edge of $\H$. Assume that $\{p, u\}$ is an edge. Then $\{p, v\}$ is not an edge, since $\deg_\H(p) <5.$ From $\deg_\H(v) \ge 3$, it follows that $\{q,v\}$ is an edge. We deduce that $\H$ contains two triangles $\{i,q,v\}$ and $\{p,u,j\}$ and three edges $\{i,u\}$, $\{q,p\}$, $\{v,j\}$, or, equivalently, $\H$ is of type 4.

If $\deg_\H(j) = 3$, then $j$ is not adjacent to one vertex in the set $\{p,q,u,v\}$. Without loss of generality, we can assume that is $v$. Since $\deg_\H(v) \ge 3$ then $v$ is adjacent to at least two vertices in $\{p,q,u\}$. Since the roles of $p,q,u$ are the same, one can assume that $\{v, q\}$ and $\{v,u\}$ are edges of $\H$. Since $\deg_\H(p) \ge 3$ then one of the vertices  $\{q,u,v\}$ is a neighbor of $p$. If $\{p, q\}$ is an edge of $\H$, then $\H$ contains two triangle $\{i,u,v\}$ and $\{p,j,q\}$ and three edges $\{i,p\}$, $\{u,j\}$, $\{v,q\}$ (Figure 4a). If $\{p, u\}$ is an edge of $\H$, then $\H$ contains two triangle $\{i,q,v\}$ and $\{p,j,u\}$ and three edges $\{i,p\}$, $\{q,j\}$, $\{v,u\}$ (Figure 4b). If both $\{p,q\}$ and $\{p,u\}$ are not edges of $\H$ then by applying condition (ii) to $p$, one gets that both $\{p,v\}$ and $\{q,u\}$ are edges of $\H$. Then $\H$ contains two triangle $\{i,p,v\}$ and $\{q,j,u\}$ and three edges $\{i,q\}$, $\{p,j\}$, $\{v,u\}$ (Figure 4c).

\begin{figure}[ht!]

\begin{tikzpicture}[scale=0.6]

\draw [very thick] (1.5,6) coordinate (f) -- (0,4) coordinate (g) ;
\draw [thin] (f) -- (1,3.5) coordinate (h) ;
\draw [very thick] (f) -- (2,3.5) coordinate (k) ;
\draw [very thick] (f) -- (3,4) coordinate (l) ;
\draw [very thick] (1.5,2) coordinate (m) -- (g) ;
\draw [very thick] (m) -- (h) ;
\draw [very thick] (m) -- (k) ;
\draw [very thick] (l) -- (h) ;
\draw [very thick] (g) -- (h) ;
\draw [very thick] (l) -- (k) ;

\fill (f) circle (2pt) ;
\fill (g) circle (2pt) ;
\fill (h) circle (2pt) ;
\fill (k) circle (2pt) ;
\fill (l) circle (2pt) ;
\fill (m) circle (2pt) ;

\draw (f) node[above] {$i$} ;
\draw (g) node[left] {$p$} ;
\draw (0.8,3.6) node[left,above] {$q$} ;
\draw (2.2,3.7) node[right,above] {$u$} ;
\draw (l) node[right] {$v$} ;
\draw (m) node[below] {$j$} ;
\draw (1.5,0) node[above] {4a};

\draw [very thick] (10,6) coordinate (a) -- (8.5,4) coordinate (b) ;
\draw [very thick] (a) -- (9.5,3.5) coordinate (c) ;
\draw [thin] (a) -- (10.5,3.5) coordinate (d) ;
\draw [very thick] (a) -- (11.5,4) coordinate (e) ;
\draw [very thick] (10,2) coordinate (x) -- (b) ;
\draw [very thick] (x) -- (c) ;
\draw [very thick] (x) -- (d) ;
\draw [very thick] (c) -- (e) ;
\draw [very thick] (b) -- (d) ;
\draw [very thick] (e) -- (d) ;

\fill (a) circle (2pt) ;
\fill (b) circle (2pt) ;
\fill (c) circle (2pt) ;
\fill (d) circle (2pt) ;
\fill (e) circle (2pt) ;
\fill (x) circle (2pt) ;

\draw (a) node[above] {$i$} ;
\draw (b) node[left] {$p$} ;
\draw (9.3,3.7) node[left,above] {$q$} ;
\draw (10.7,3.7) node[right,above] {$u$} ;
\draw (e) node[right] {$v$} ;
\draw (x) node[below] {$j$} ;
\draw (10,0) node[above] {4b} ;

\draw [very thick] (17.5,6) coordinate (u) -- (16,4) coordinate (o) ;
\draw [very thick] (u) -- (17,3.5) coordinate (p) ;
\draw [thin] (u) -- (18,3.5) coordinate (q) ;
\draw [very thick] (u) -- (19,4) coordinate (r) ;
\draw [very thick] (17.5,2) coordinate (v) -- (o) ;
\draw [very thick] (v) -- (p) ;
\draw [very thick] (v) -- (q) ;
\draw [thin] (p) -- (r) ;
\draw [very thick] (o) -- (r) ;
\draw [very thick] (p) -- (q) ;
\draw [very thick] (r) -- (q) ;

\fill (u) circle (2pt) ;
\fill (v) circle (2pt) ;
\fill (o) circle (2pt) ;
\fill (p) circle (2pt) ;
\fill (q) circle (2pt) ;
\fill (r) circle (2pt) ;

\draw (u) node[above] {$i$} ;
\draw (o) node[left] {$p$} ;
\draw (17.1,3.5) node[left] {$q$} ;
\draw (18.2,3.5) node[right,below] {$u$} ;
\draw (r) node[right] {$v$} ;
\draw (v) node[below] {$j$} ;
\draw (17.5,0) node[above] {4c} ;

\end{tikzpicture}
\caption{}
\end{figure}

\noindent \underline{Subcase 1.3}: $\max \deg_\H(i) = 3$. It follows that $\deg_\H(i) = 3$ for every vertex $i \in \H$. Due to \cite[Lemma 2.7]{HLT} any 4-saturating graph contains an odd cycle of length $\le 7.$ Hence, $\H$ contains a triangle or a pentagon. If $\H$ contains an induced pentagon, since $\deg_\H(i) = 3$ for all $i \in \H$, then all the vertices of this pentagon must be adjacent to the remain vertex not in this pentagon. But in this case the degree of this vertex is $5$, this is a contradiction to the assumption that $\max \deg_\H(i) = 3$. It follows that $\H$ has no induced pentagon, but $\H$ has a triangle. 

Let $T_1:= \{u,v,w\}$ be a triangle. We claim that the three remain vertices are pairwise adjacent. Assume the opposition that they are not pairwise adjacent. Let $i,j$ be such a pair of non-adjacent vertices. Since $\deg_\H(i)=\deg_\H(j) =3$ then both $i$ and $j$ is connected to at least two vertices of the triangle $\{u,v,w\}$. Then at least one vertex of $\{u,v,w\}$ is connected to both $i$ and $j$, and hence it has degree $\ge 4,$ a contradiction to the assumption that $\max \deg_\H(i) = 3$. So the remain vertices form a triangle, denoted by $T_2.$ Moreover, the condition $\deg_\H(i) =3 \ \forall i$ implies also that any vertex of $T_1$ is connected to exactly one vertex of $T_2$, or, equivalently, $\H$ is of type 4. 

\begin{figure}[ht!]

\begin{tikzpicture}[scale=0.4]

\draw [thick] (0,0) coordinate (a) -- (0,3) coordinate (b) ;
\draw [thick] (0,0) coordinate (a) -- (2,1.5) coordinate (c) ;
\draw [thick] (0,0) coordinate (a) -- (7,0) coordinate (d) ;
\draw [thick]  (0,3) coordinate (b)-- (7,3) coordinate (e);
\draw [thick]  (0,3) coordinate (b)-- (2,1.5) coordinate (c);
\draw [thick] (2,1.5) coordinate (c)-- (5,1.5) coordinate (f);
\draw [thick]  (7,0) coordinate (d)-- (7,3) coordinate (e);
\draw [thick] (7,0) coordinate (d)-- (5,1.5) coordinate (f);
\draw [thick] (7,3) coordinate (e)-- (5,1.5) coordinate (f);

\fill (a) circle (2pt);
\fill (b) circle (2pt);
\fill (c) circle (2pt);
\fill (d) circle (2pt);
\fill (e) circle (2pt);
\fill (f) circle (2pt);

\draw (0,0) node[left,below] {$u$};
\draw (0,3) node[left,above] {$v$};
\draw (2,1.5) node[right, below] {$w$};
\draw (7,0) node[right, below]{$i$};
\draw (7,3) node[right,above]{$j$};
\draw (5,1.5) node[left,below]{$p$};

\end{tikzpicture}
\caption{ }
\end{figure}

\noindent \underline{Case 2}: $n = 7$. \par

If $\H$ is disconnected, then \cite[Lemma 2.8]{HLT} implies that $\H$ is the union of a $2$-saturating graph and a $3$-saturating graph. Therefore, $\H$ is the union of a triangle and a $K_4$ (that is of Type 7). So we can assume that $\H$ is connected. Due to \cite[Lemma 2.7]{HLT}, $\H$ contains either a triangle, or a pentagon, or a $C_7$. If $\H$ contains a $C_7$, then $\H$ is of type 8. \par

Next, we consider only the case where $\H$ contains no $C_7.$ Let $M$ be a maximum matching of $\H$. Since $\nu(\H)=3$ and $n=7$, there exists a unique vertex $i$ not in $M$. Due to Lemma \ref{berge}, there is a $M$-augmenting odd cycle $P_i$ that starts and ends at $i$. Since $\H$ contains no $C_7$ then $P_i$ is either a triangle or a pentagon. We distinguish the following subcases: \par 

\noindent \underline{Subcase 2.1}: {\it There is an $M$-augmenting pentagon $P_i$ that starts and ends at $i$.} In this case, $P_i$ contains two matching edges of $M$. Assume that $P_i:=\{i,b,c,d,e\}$ and $\{u,v\}$ is the remain matching edge of $M$ which is not in $P_i$. Since $\deg_{\H}(j) \ge 2 \ \forall j \in \H$ then both $u$ and $v$ is connected to $P_i$. The condition that $\H$ contains no $C_7$ implies that either both $u$ and $v$ are connected to a vertex of $P_i$, or $u$ and $v$ are connected to two non-adjacent vertices of $P_i$. Therefore, $\H$ is of type 9 or type 10.

\noindent \underline{Subcase 2.2}: {\it There is no $M$-augmenting pentagon that starts and ends at $i$.} Then $P_i$ is a triangle. Let $\{u,v\}$ and $\{p, q\}$ denote the edges of $M$ not in $P_i$. If $\{u,v\}$ and $\{p,q\}$ are non-adjacent then both $\{p,q\}$ and $\{u,v\}$ are connected to $P_i$ (since $\H$ is connected). Moreover, since there is no $M$-augmenting pentagon that starts and ends at $i$, both $p$ and $q$ is adjacent to a vertex of $P_i$, and similarly to $u$ and $v.$ But $\{u,v\}$ and $\{p,q\}$ are non-adjacent, then $p,q,u,v$ are adjacent to the same vertex of $P_i.$ Hence, $\H$ is of Type 11.

If $\{p,q\}$ and $\{u,v\}$ are adjacent, then we can assume, without loss of generality, that is $\{q,u\}$ (see Figure 6). We have the following subcases:

{\it $\bullet$ $p$ or $v$ is adjacent to the triangle $P_a$.} We can assume that is $p$ and assume that $\{p,x\}$ is in $\H$ for some $x \in P_a$ (Figure 6a). If $v$ is adjacent to $P_a$ then either $\H$ is contains a $C_7$ or $\H$ is of type 9. If $v$ is not adjacent to $P_a$ then the condition $\nu(\H - N_\H(v)) \ge 4 - \deg_\H(v)$ implies that $v$ has at least two neighbors, and hence either $\{v,p\}$ or $\{v,q\}$ is in $\H$. If $\{v,q\} \in \H$  then $\H$ is of type 12. If $\{v,q\} \notin \H$ then $\{v,p\} \in \H$. But the condition  $\nu(\H - N_\H(q)) \ge 4 - \deg_\H(q)$ implies that 
$$\deg_\H(q) \ge 4 - \nu(\H- N_\H(q)) \ge 3.$$
Its follows that $q$ is adjacent to $P_a$. Therefore, $\H$ is either contains a $C_7$ or of type 9. 
\begin{figure}[ht!]

\begin{tikzpicture}[scale=0.6]
\draw [thin] (0,0) coordinate (b) -- (2,0) coordinate (c) ;
\draw [thin] (b) -- (1,2) coordinate (a) ;
\draw [thin] (a) -- (c);
\draw [thin] (a) -- (3,2) coordinate (p) ;
\draw [very thick] (p) -- (5,2) coordinate (q) ;
\draw [thin] (4,0)coordinate (u) -- (q) ;
\draw [thin] (6,0) coordinate (v) -- (u);

\fill (a) circle (2pt);
\fill (b) circle (2pt);
\fill (c) circle (2pt);
\fill (p) circle (2pt);
\fill (q) circle (2pt);
\fill (u) circle (2pt);
\fill (v) circle (2pt);

\draw (3,-1) node[below] {6a};
\draw (p) node[above] {$p$};
\draw (q) node[above] {$q$};
\draw (u) node[below] {$u$};
\draw (v) node[below] {$v$};
\draw (a) node[above] {$x$};

\draw [thin] (9,0) coordinate (b) -- (11,0) coordinate (c) ;
\draw [thin] (b) -- (10,2) coordinate (a) ;
\draw [thin] (a) -- (c);
\draw [thin] (a) -- (12,2) coordinate (p) ;
\draw [very thick] (p) -- (14,2) coordinate (q) ;
\draw [thin] (13,0)coordinate (u) -- (p) ;
\draw [thin] (15,0) coordinate (v) -- (u);

\fill (a) circle (2pt);
\fill (b) circle (2pt);
\fill (c) circle (2pt);
\fill (p) circle (2pt);
\fill (q) circle (2pt);
\fill (u) circle (2pt);
\fill (v) circle (2pt);

\draw (12,-1) node[below] {6b};
\draw (p) node[above] {$q$};
\draw (q) node[above] {$p$};
\draw (u) node[below] {$u$};
\draw (v) node[below] {$v$};
\draw (a) node[above] {$x$};

\end{tikzpicture}
\caption{}
\end{figure}

{\it $\bullet$ Both $p$ and $v$ are not adjacent to the triangle $P_a$.} Since $\H$ is connected then either $q$ or $u$ is connected to $P_a$. We can assume that is $q$ (Figure 6b). From the condition $\nu(\H - N_\H(p)) \ge 4 - \deg_\H(p)$ it follows that $p$ has at least two neighbors. But $p$ is not adjacent to $P_a$. Hence, $p$ is either adjacent to $u$ or $v$. Now, by changing the roles of $p$ and $q$, we return to the case above.\par

\noindent \underline{Case 3}: $n = 8$. \cite[Lemma 2.7]{HLT} implies that $\H$ contains either a triangle, or a pentagon, or a $C_7$. If $\H$ contains a $C_7$ then $\H$ is connected, and $\H$ contains the union of a $C_7$ and an edge. Hence, $\nu(\H) =4$. This is a contradiction. Therefore, $\H$ contains no $C_7.$

If $\H$ is disconnected, then \cite[Lemma 2.8]{HLT} implies that $\H$ is a disjoint union of a 2-saturating graph and a 3-saturating graph. Due to \cite{TeraiTrung} a 2-saturating graph is a triangle; and due to Theorem \ref{3-saturating} a 3-saturating graph with $5$ vertices is either a pentagon or a union of two triangles meeting at a vertex. It follows that $\H$ is either of type 5 or type 6. 

Next, we assume that $\H$ is connected. Since $\nu(\H)=3$, we can assume that $M:=\left\{\{m,n\}, \{p,q\}, \{u,v\}\right\}$ is a maximum matching of $\H$, and $i,j$ are the two remain vertices not in $M$. It is obvious that $\{i,j\}$ is not an edge of $\H$. Due to Lemma \ref{berge}, there are two disjoint $M$-augmenting odd cycles, denoted by $P_i$ and $P_j$, that start and end at $i$ and resp. at $j$. Since $\H$ contains no $C_7$ then $P_i$ and $P_j$ is either triangles or pentagons. Since $n=8$ then $P_i$ and $P_j$ can not be two pentagons. We distinguish two subcases:

\noindent \underline{Subcase 3.1}: {\it $P_i$ and $P_j$ are a triangle and a pentagon.} Then the connected graph on $8$ vertices $\H$ must contain a pentagon and a triangle connected by an edge. Hence, $\nu(\H) \ge 4$, a contradiction. 

\noindent \underline{Subcase 3.2}: {\it $P_i$ and $P_j$ are two disjoint triangles.} Without loss of generality, we can assume that the triangles are $T_1:=\{i,m,n\}$ and $T_2:=\{j,p,q\}$. If $T_1$ and $T_2$ are connected by an edge then the induced subgraph of $\H$ on the set $\{i,m,n,j,p,q\}$ has matching number~3. Since the vertex set of this subgraph and $\{u,v\}$ are disjoint, $\nu(\H) = 3 +1 =4$, this is a contradiction to condition (i). Therefore, $T_1$ and $T_2$ are non-adjacent. Because $\H$ is connected then both $T_1$ and $T_2$ are adjacent to the matching edge $\{u,v\}$. Moreover, since $\deg_\H(u) \ge 2$ and $\deg_\H(v) \ge 2$ then there exists two disjoint edges such that one edge connects $T_1$ and $\{u,v\}$ and the other connects $T_2$ and $\{u,v\}$. Therefore, $\H$ contains two triangles connected by a path of length 3 (see Figure 7).

\begin{figure}[ht!]

\begin{tikzpicture}[scale=0.6]
\draw [thin] (1,0) coordinate (g) -- (2,2) coordinate (h) ;
\draw [thin] (2,2) coordinate (h) -- (3,0) coordinate (i) ;
\draw [thin] (i) -- (g) ; 

\draw [very thick] (7,0) coordinate (c) -- (5,0) coordinate (a) ; 

\draw [thin] (9,0) coordinate (d) -- (10,2) coordinate (e);
\draw [thin] (e) -- (11,0) coordinate (f);
\draw [thin] (f) -- (d);

\draw [thin] (a) -- (i) ;
\draw [thin] (c) -- (d) ;

\fill (a) circle (2pt);
  \fill (c) circle (2pt);
  \fill (d) circle (2pt);
  \fill (e) circle (2pt);
  \fill (f) circle (2pt);
   \fill (g) circle (2pt);
    \fill (h) circle (2pt);
     \fill (i) circle (2pt);

\draw (2,0) node[below] {$T_1$};
\draw (5,0) node[below] {$u$};
\draw (7,0) node[below] {$v$} ;
\draw (10,0) node[below] {$T_2$};

\end{tikzpicture}
\caption{}
\end{figure}

Hence, $\nu(\H) \ge 4$, a contradiction to our assumption.
\par

\noindent \underline{Case 4}: $n = 9$. Then there exists a maximum matching $M$ with 3 disjoint edges $e_1, e_2, e_3$ and 3 different vertices $u_1, u_2, u_3$ not in $M$. Lemma \ref{berge} implies that there are 3 disjoint odd cycles starting at $u_1, u_2, u_3$. Hence, these cycles must be triangles. The condition (i) implies that these triangles are non-adjacent, and it follows that $\H$ is the union of these three disjoint triangles. \par

The proof of Theorem \ref{squarefree} is now complete.\par 
\end{proof}

By Proposition \ref{polar}, the polarization of a 4-saturating weighted graph is also $4$-saturating. Hence, we can use Theorem \ref{squarefree} to classify all 4-saturating weighted graphs. 

\begin{center}
\begin{picture}(0,0)%
\includegraphics{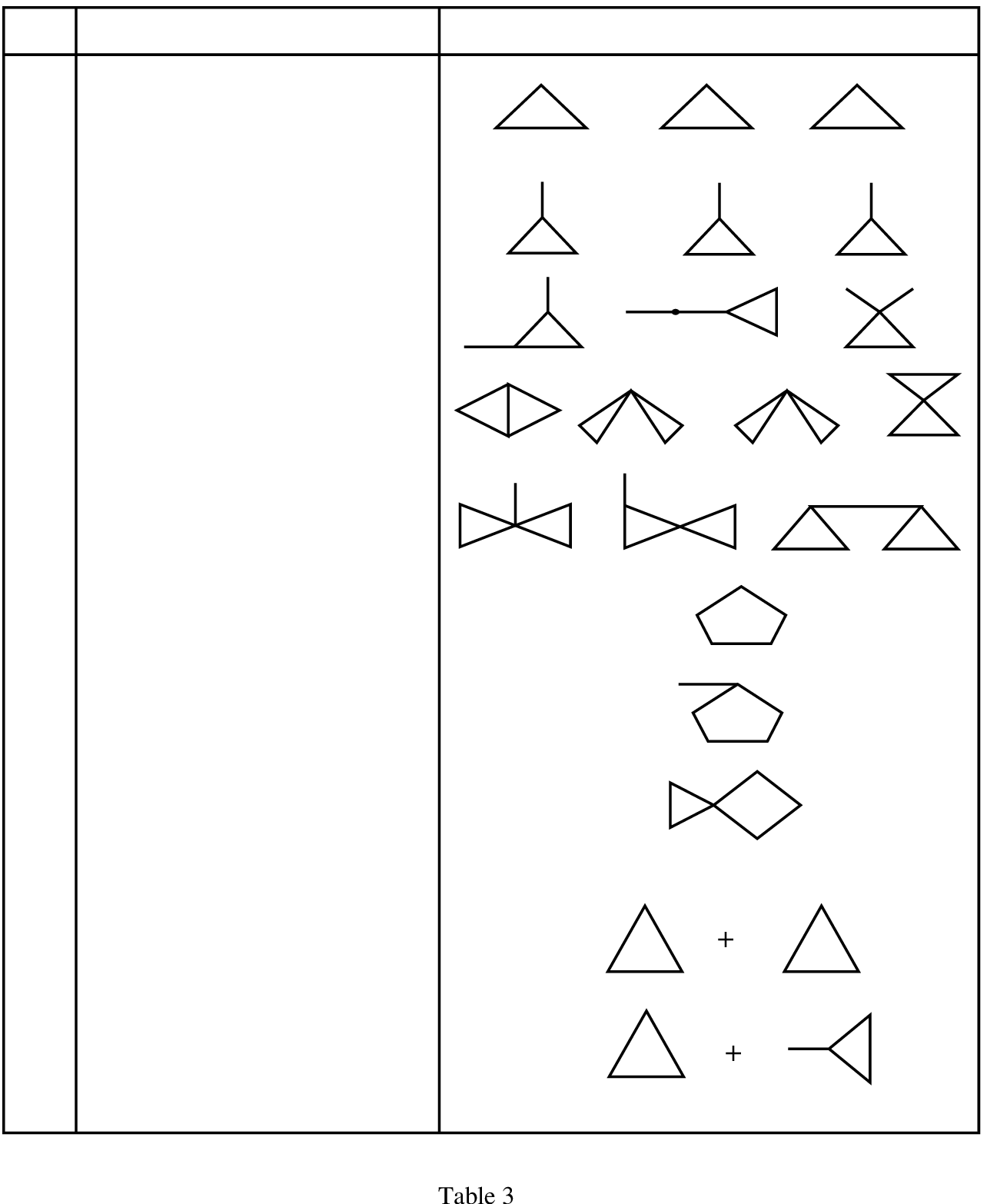}%
\end{picture}%
\setlength{\unitlength}{4144sp}%
\begingroup\makeatletter\ifx\SetFigFont\undefined
\def\x#1#2#3#4#5#6#7\relax{\def\x{#1#2#3#4#5#6}}%
\expandafter\x\fmtname xxxxxx\relax \def\y{splain}%
\ifx\x\y   
\gdef\SetFigFont#1#2#3{%
  \ifnum #1<17\tiny\else \ifnum #1<20\small\else
  \ifnum #1<24\normalsize\else \ifnum #1<29\large\else
  \ifnum #1<34\Large\else \ifnum #1<41\LARGE\else
     \huge\fi\fi\fi\fi\fi\fi
  \csname #3\endcsname}%
\else
\gdef\SetFigFont#1#2#3{\begingroup
  \count@#1\relax \ifnum 25<\count@\count@25\fi
  \def\x{\endgroup\@setsize\SetFigFont{#2pt}}%
  \expandafter\x
    \csname \romannumeral\the\count@ pt\expandafter\endcsname
    \csname @\romannumeral\the\count@ pt\endcsname
  \csname #3\endcsname}%
\fi
\fi\endgroup
\begin{picture}(5844,7207)(144,-6290)

\put(3336,400){\makebox(0,0)[lb]{\smash{\SetFigFont{10}{12.0}{rm}$2$}}}
\put(4317,400){\makebox(0,0)[lb]{\smash{\SetFigFont{10}{12.0}{rm}$3$}}}
\put(5199,400){\makebox(0,0)[lb]{\smash{\SetFigFont{10}{12.0}{rm}$3$}}}
\put(2976,110){\makebox(0,0)[lb]{\smash{\SetFigFont{10}{12.0}{rm}$2$}}}
\put(3654,110){\makebox(0,0)[lb]{\smash{\SetFigFont{10}{12.0}{rm}$2$}}}
\put(4647,110){\makebox(0,0)[lb]{\smash{\SetFigFont{10}{12.0}{rm}$2$}}}
\put(3947,110){\makebox(0,0)[lb]{\smash{\SetFigFont{10}{12.0}{rm}$2$}}}
\put(4840,110){\makebox(0,0)[lb]{\smash{\SetFigFont{10}{12.0}{rm}$3$}}}
\put(5524,110){\makebox(0,0)[lb]{\smash{\SetFigFont{10}{12.0}{rm}$1$}}}

\put(5226,-240){\makebox(0,0)[lb]{\smash{\SetFigFont{10}{12.0}{rm}$2$}}}
\put(4322,-240){\makebox(0,0)[lb]{\smash{\SetFigFont{10}{12.0}{rm}$3$}}}
\put(3270,-240){\makebox(0,0)[lb]{\smash{\SetFigFont{10}{12.0}{rm}$3$}}}
\put(5549,-600){\makebox(0,0)[lb]{\smash{\SetFigFont{10}{12.0}{rm}$2$}}}
\put(5025,-600){\makebox(0,0)[lb]{\smash{\SetFigFont{10}{12.0}{rm}$2$}}}
\put(4122,-600){\makebox(0,0)[lb]{\smash{\SetFigFont{10}{12.0}{rm}$2$}}}
\put(3073,-600){\makebox(0,0)[lb]{\smash{\SetFigFont{10}{12.0}{rm}$2$}}}

\put(5342,-900){\makebox(0,0)[lb]{\smash{\SetFigFont{10}{12.0}{rm}$3$}}}
\put(3270,-940){\makebox(0,0)[lb]{\smash{\SetFigFont{10}{12.0}{rm}$2$}}}
\put(3114,-1140){\makebox(0,0)[lb]{\smash{\SetFigFont{10}{12.0}{rm}$2$}}}
\put(4387,-920){\makebox(0,0)[lb]{\smash{\SetFigFont{10}{12.0}{rm}$2$}}}
\put(4120,-920){\makebox(0,0)[lb]{\smash{\SetFigFont{10}{12.0}{rm}$2$}}}

\put(4776,-1400){\makebox(0,0)[lb]{\smash{\SetFigFont{10}{12.0}{rm}$2$}}}
\put(3856,-1400){\makebox(0,0)[lb]{\smash{\SetFigFont{10}{12.0}{rm}$2$}}}
\put(3126,-1400){\makebox(0,0)[lb]{\smash{\SetFigFont{10}{12.0}{rm}$2$}}}
\put(5870,-1690){\makebox(0,0)[lb]{\smash{\SetFigFont{10}{12.0}{rm}$2$}}}
\put(5325,-1690){\makebox(0,0)[lb]{\smash{\SetFigFont{10}{12.0}{rm}$2$}}}
\put(3133,-1850){\makebox(0,0)[lb]{\smash{\SetFigFont{10}{12.0}{rm}$2$}}}
\put(4414,-1620){\makebox(0,0)[lb]{\smash{\SetFigFont{10}{12.0}{rm}$2$}}}

\put(3730,-2140){\makebox(0,0)[lb]{\smash{\SetFigFont{10}{12.0}{rm}$2$}}}
\put(3157,-2390){\makebox(0,0)[lb]{\smash{\SetFigFont{10}{12.0}{rm}$2$}}}
\put(4926,-2120){\makebox(0,0)[lb]{\smash{\SetFigFont{10}{12.0}{rm}$2$}}}

\put(4180,-2755){\makebox(0,0)[lb]{\smash{\SetFigFont{10}{12.0}{rm}$2$}}}
\put(4522,-2575){\makebox(0,0)[lb]{\smash{\SetFigFont{10}{12.0}{rm}$2$}}}

\put(4503,-3156){\makebox(0,0)[lb]{\smash{\SetFigFont{10}{12.0}{rm}$2$}}}

\put(4328,-4070){\makebox(0,0)[lb]{\smash{\SetFigFont{10}{12.0}{rm}$2$}}}

\put(5026,-4440){\makebox(0,0)[lb]{\smash{\SetFigFont{10}{12.0}{rm}$1$}}}
\put(5291,-4900){\makebox(0,0)[lb]{\smash{\SetFigFont{10}{12.0}{rm}$2$}}}
\put(4715,-4900){\makebox(0,0)[lb]{\smash{\SetFigFont{10}{12.0}{rm}$2$}}}

\put(4979,-5320){\makebox(0,0)[lb]{\smash{\SetFigFont{10}{12.0}{rm}$2$}}}

\put(305,-5461){\makebox(0,0)[lb]{\smash{\SetFigFont{10}{12.0}{rm}10}}}
\put(306,-4673){\makebox(0,0)[lb]{\smash{\SetFigFont{10}{12.0}{rm}9}}}
\put(306,-3974){\makebox(0,0)[lb]{\smash{\SetFigFont{10}{12.0}{rm}8}}}
\put(306,-3298){\makebox(0,0)[lb]{\smash{\SetFigFont{10}{12.0}{rm}7}}}
\put(306,-2800){\makebox(0,0)[lb]{\smash{\SetFigFont{10}{12.0}{rm}6}}}
\put(306,-2240){\makebox(0,0)[lb]{\smash{\SetFigFont{10}{12.0}{rm}5}}}
\put(306,-1694){\makebox(0,0)[lb]{\smash{\SetFigFont{10}{12.0}{rm}4}}}
\put(306,-953){\makebox(0,0)[lb]{\smash{\SetFigFont{10}{12.0}{rm}3}}}
\put(306,-412){\makebox(0,0)[lb]{\smash{\SetFigFont{10}{12.0}{rm}2}}}
\put(306,175){\makebox(0,0)[lb]{\smash{\SetFigFont{10}{12.0}{rm}1}}}

\put(626,-5765){\makebox(0,0)[lb]{\smash{\SetFigFont{10}{12.0}{rm}weight $2$}}}
\put(626,-5597){\makebox(0,0)[lb]{\smash{\SetFigFont{10}{12.0}{rm}edge meeting at a vertex of}}}
\put(626,-5428){\makebox(0,0)[lb]{\smash{\SetFigFont{10}{12.0}{rm}spanned by a triangle and an}}}
\put(626,-5259){\makebox(0,0)[lb]{\smash{\SetFigFont{10}{12.0}{rm}simple triangle and a graph}}}
\put(626,-5091){\makebox(0,0)[lb]{\smash{\SetFigFont{10}{12.0}{rm}$\Omega$ is a disjoint union of a}}}

\put(626,-4803){\makebox(0,0)[lb]{\smash{\SetFigFont{10}{12.0}{rm}with weight $(2,2,1)$}}}
\put(626,-4634){\makebox(0,0)[lb]{\smash{\SetFigFont{10}{12.0}{rm}simple triangle and a triangle}}}
\put(626,-4456){\makebox(0,0)[lb]{\smash{\SetFigFont{10}{12.0}{rm}$\Omega$ is a disjoint union of a}}}

\put(626,-4182){\makebox(0,0)[lb]{\smash{\SetFigFont{10}{12.0}{rm}at a vertex of weight $2$}}}
\put(626,-4035){\makebox(0,0)[lb]{\smash{\SetFigFont{10}{12.0}{rm}a rectangle and a triangle }}}
\put(626,-3868){\makebox(0,0)[lb]{\smash{\SetFigFont{10}{12.0}{rm}$\Omega$ is spanned by the union of}}}

\put(626,-3600){\makebox(0,0)[lb]{\smash{\SetFigFont{10}{12.0}{rm}ing at a vertex of weight $2$}}}
\put(626,-3425){\makebox(0,0)[lb]{\smash{\SetFigFont{10}{12.0}{rm}a pentagon and an edge meet-}}}
\put(626,-3250){\makebox(0,0)[lb]{\smash{\SetFigFont{10}{12.0}{rm}$\Omega$ is spanned by a union of}}}

\put(626,-2900){\makebox(0,0)[lb]{\smash{\SetFigFont{10}{12.0}{rm}with weight $(2,2,1,1,1)$}}}
\put(626,-2720){\makebox(0,0)[lb]{\smash{\SetFigFont{10}{12.0}{rm}$\Omega$ is spanned by a pentagon}}}

\put(626,-2350){\makebox(0,0)[lb]{\smash{\SetFigFont{10}{12.0}{rm}two triangles and an edge}}}
\put(626,-2160){\makebox(0,0)[lb]{\smash{\SetFigFont{10}{12.0}{rm}$\Omega$ is spanned by a union of }}}

\put(626,-1732){\makebox(0,0)[lb]{\smash{\SetFigFont{10}{12.0}{rm}two triangles}}}
\put(626,-1547){\makebox(0,0)[lb]{\smash{\SetFigFont{10}{12.0}{rm}$\Omega$ is spanned by a union of}}}

\put(626,-1060){\makebox(0,0)[lb]{\smash{\SetFigFont{10}{12.0}{rm}triangle and 2 edges}}}
\put(626,-894){\makebox(0,0)[lb]{\smash{\SetFigFont{10}{12.0}{rm}$\Omega$ is spanned by a union of a }}}

\put(626,-433){\makebox(0,0)[lb]{\smash{\SetFigFont{10}{12.0}{rm}triangle and an edge}}}
\put(626,-239){\makebox(0,0)[lb]{\smash{\SetFigFont{10}{12.0}{rm}$\Omega$ is spanned by a union of a }}}

\put(626,181){\makebox(0,0)[lb]{\smash{\SetFigFont{10}{12.0}{rm}$(2,2,2)$ or $(3,2,2)$ or $(3,3,1)$}}}
\put(626,362){\makebox(0,0)[lb]{\smash{\SetFigFont{10}{12.0}{rm}$\Omega$ is a triangle with weight}}}

\put(193,600){\makebox(0,0)[lb]{\smash{\SetFigFont{10}{12.0}{rm}Type}}}
\put(1306,600){\makebox(0,0)[lb]{\smash{\SetFigFont{10}{12.0}{rm}Description}}}
\put(4073,600){\makebox(0,0)[lb]{\smash{\SetFigFont{10}{12.0}{rm}Figure}}}

\end{picture}

\end{center}

\begin{Theorem} \label{4-sat weighted graphs} 
A weighted graph $\H$ is 4-saturating if and only if $\H$ belongs to the types of Table 2 and Table 3. 
\end{Theorem}

\begin{proof} We can easily verify that the connected graphs belongs to the types of Table 3 are $4$-saturating. 

Now, let $\H$ be a $4$-saturating graph. We will prove that $\H$ belongs to the types of Table 2 and Table 3. If $\H$ is a simple graph, i.e. all vertices of $\H$ are of weight 1, then Proposition \ref{squarefree} implies that $\H$ is one of the types of Table 2. If $\H$ is a weighted graph with some vertex of weight at least 2, then as in Proposition \ref{polar} implies that $\p(\H)$ must belongs to the types of Table 2. Hence, $\H$ can be obtained from a graph belonging to Table 2 by collapsing some non-adjacent vertices with the same neighborhood. Let $W$ denote the vertex set of $\p(\H)$. Then $|W| \neq 5,$ and $|W|\neq 9$, since in these cases, $\p(\H)$ is either a $K_5$ or a union of three non-adjacent triangles, responsively, and then it has no non-adjacent vertices which have the same neighborhood.   

We will treat all the possible cases of $p(\H)$ according to $|W|$. We have:

\noindent \underline{Case 1}: $|W|=6$. By collapsing some non-adjacent vertices with the same neighborhood, we obtain either a triangle with weight vector $(2,2,2)$, or two triangles sharing an edge of weight $(2,2)$ , or a union of two triangles meeting at a vertex of weight 2.

\noindent \underline{Case 2}: $|W|=7$. By collapsing some non-adjacent vertices with the same neighborhood, we obtain either a triangle with weight vector $(3,2,2)$, or $(3,3,1)$; or a union of a triangle and an edge of weight vector as in Type 2 of Table 3; or a union of a triangle and two edges as in Type 3 of Table 3; or a union of a triangles $(1,1,1)$ and a triangle $(2,2,1)$ sharing a common vertex (that is the cases 3 and 4, Type 4 of Table 3); or a pentagon with an edge of weight $(2,2)$ ; or a union of a pentagon and an edge meeting at a vertex of weight 2 (that is of Type 7 of Table 3); or a union of two triangles and an edge as in Type 5 of Table 3; or union of a triangle and a rectangle meeting at a vertex of weight 2 (that is of Type 8 of Table 3).

\noindent \underline{Case 3}: $|W|=8$. $\p(\H)$ must be a disjoint union of a triangle and a simple $3$-saturating graph. Hence $\H$ must be of Type 9 or Type 10 of Table 3.
\end{proof}


\section{Saturation and associated prime ideals of $I^4$}

In this section, we first describe all monomials of the saturation of $I^4$,
where $I$ is the edge ideal of a simple graph $\G$.
Then we will give a combinatorial characterization of the associated primes of $I^4.$

\begin{Theorem} \label{monomial} 
Let $\a \in \NN^n$. Then $\xx^\a \in \widetilde{I^4} \setminus I^4$ if only if $V_\a$ is a dominating set of $\G$ and one of the following conditions is satisfied:\par

{\rm{(i)}} $\G_\a$ is the disjoint union of a triangle and a $K_4$ (that is Type 7\ of Table 2) such that every $i \in V \setminus V_\a$ is adjacent to either the triangle or adjacent to every triangles of the $K_4$;\par

{\rm{(ii)}} $\G_\a$ is spanned by two disjoint triangles connected by a path of length 2 (that is Type 12 of Table 2) such that either the endpoints of the path are adjacent in $\G_\a$ or every vertex in $V \setminus V_\a$ is adjacent to at least one triangle;\par 

{\rm{(iii)}} $\G_\a$ belongs to the types of Table 2 minus Type 7 and Type 12 or belongs to the type of Table 3, and for any $i \in V \setminus V_\a$, $\deg_{\a}(i) > 7 - \deg(\xx^\a).$

\end{Theorem}
 
\begin{proof} Due to Theorem \ref{satthen}, $\xx^\a \in  \tilde{I^4} \setminus I^4$ if and only if $\G_\a$ is $4$-saturating and $\nu(\G_\a  - N_\a(i)) \ge 4 - \deg_\a(i)$ for all $i \in V \setminus V_\a$. From this it follows that $\G_\a$ is $4$-saturating and $V_\a$ is a dominating set of $\G$. Therefore, $\G_\a$ belongs to the types of Table 2 and Table 3. Hence, $5 \le \deg(\xx^\a) \le 9.$ \par

We distinguish the following cases \par

\noindent {\underline{Case 1}}: $\deg(\xx^\a)=5$. Then $\G_\a$ is a $K_5$, and the condition $\nu(\G_\a  - N_\a(i)) \ge 4 - \deg_\a(i)$ implies that $i$ must be adjacent to at least 3 vertices of $\G_\a$ for all $i \in V\setminus V_\a.$ Hence,
$$\deg_\a(i) \ge 3 > 7-5=7-\deg(\xx^a) \ \forall i \in V \setminus V_\a.$$

\noindent {\underline{Case 2}}: $\deg(\xx^\a)=6$. Then the condition $\nu(\G_\a  - N_\a(i)) \ge 4 - \deg_\a(i)$ implies that $\deg_\a(i) \ge 2$ for all $i \in V \setminus V_\a.$ Hence 
$$\deg_\a(i) \ge 2 > 7-6=7-\deg(\xx^a) \ \forall i \in V \setminus V_\a.$$

\noindent {\underline{Case 3}}: $\deg(\xx^\a) \ge 7$. Then the condition $\nu(\G_\a  - N_\a(i)) \ge 4 - \deg_\a(i)$ implies that $i$ must be adjacent to some vertex of $\G_\a$ for all $i \in V \setminus V_\a,$ or 
$$\deg_\a(i) \ge 1> 7-\deg(\xx^a) \ \forall i \in V \setminus V_\a.$$ 

Moreover, in the case $\G_\a$ is a disjoint union of a triangle and a $K_4$, if there is a vertex $i \in V \setminus V_\a$ such that $i$ is not adjacent to the triangle, then $\nu(\G_\a - N_\a(i)) = \nu(K_4 - N_{K_4}(i)) +1$. Hence, the condition $\nu(\G_\a  - N_\a(i)) \ge 4 - \deg_\a(i)$ implies that $\deg_{K_4}(i) = \deg_\a(i) \ge 3 - \nu(K_4 - N_{K_4}(i)) \ge 2$. Therefore, $i$ must be adjacent to every triangles of the $K_4$.

In the case $\G_\a$ is the graph of two disjoint triangles connected by a path of length 2, we denote by $v$ the middle point of the path. 

If the two triangles are adjacent then either the endpoints of the path are adjacent, or $\G_\a$ belongs to either the type 8 or type 9 of Table 2.

If the two triangles are non-adjacent then $\nu(\G_\a - N_\a(v)) = 2$. Therefore, for every vertex $i \in V \setminus V_\a$ such that $i$ is adjacent to $v$, one has $\deg_\a(i) \ge 4- \nu(\G_\a  - N_\a(i)) \ge 4 - \nu(\G_\a - N_\a(v)) =2$. From this it follows that $i \in V \setminus V_\a$ must be adjacent to one of the triangles.  \par

Conversely, it is easy to verify that if $\G_\a$ is one of graphs satisfying one of the conditions in the theorem then $\G_\a$ is a $4$-saturating, and $\nu(\G_\a  - N_\a(i)) \ge 4 - \deg_\a(i)$ for every $i \in V \setminus V_\a$. Hence, $\xx^\a \in  \tilde{I^4} \setminus I^4$.
\end{proof}

Applying the above theorem and Theorem \ref{monomial}, we can characterize the associated primes of $I^4$ as follows:
 
\begin{Theorem} \label{asso 4} Let $F$ be a cover of $\G$. Then $P_F$ is an associated prime of $I^4$ if and only if $F$ is a minimal cover or $F$ is minimal among the covers containing the closed neighborhood of the vertex set of one of subgraphs in Table 4.
\end{Theorem}

\begin{proof} It is well-known that $P_F$ is a minimal prime of $I^4$ if and only if $F$ is a minimal cover of $\G$. \par

Next, we consider only the case where $F$ is not a minimal cover of $\G$. One has core$(F) \neq \emptyset$. 
Let $S := k[x_i|\ i \in \mbox{core}(F)]$ and $J := I(\G_{\rm{core}(F)})$.\par
  
By Theorem \ref{associated prime It}, $P_F$ is an associated prime of $I^4$ if and only if $F$ is a minimal  among the covers containing $N[V_\a]$ for some $\a \in \NN^n$ such that $\G_\a$ is $4$-saturating and $\nu(\G_\a  - N_\a(i)) \ge 4 - \deg_\a(i)$ for all $i \in \rm{core}(F) \setminus V_\a$. By Theorem \ref{monomial}, this is the case if and only if $\G_\a$ satisfies one of the conditions in Theorem \ref{monomial}.

We will prove that the base graph $\G_{V_\a}$ of $\G_\a$ is one of the subgraphs of Table 4. We treat all the possible types of $\G_\a$.\par

\begin{center}
\begin{picture}(0,0)%
\includegraphics{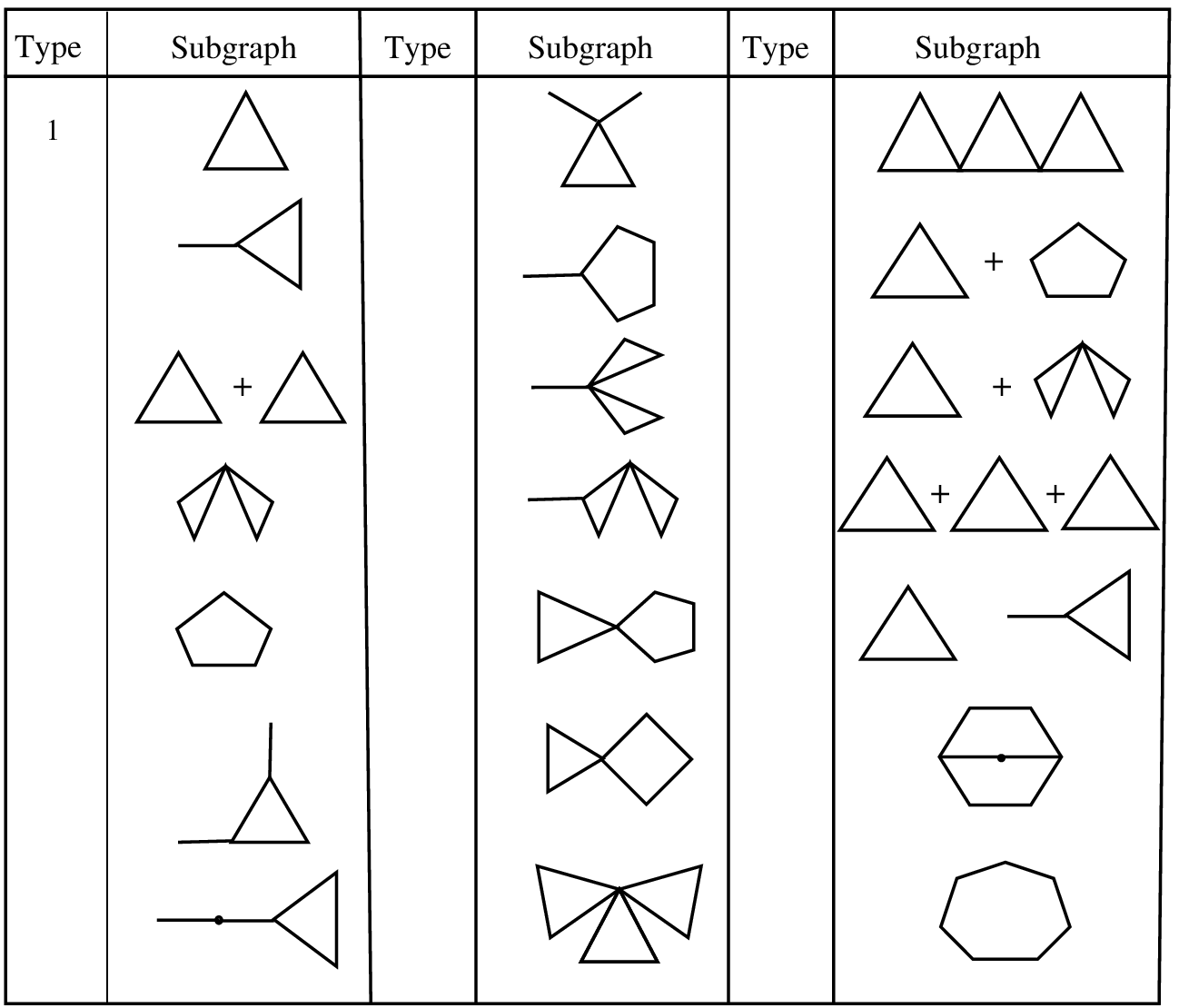}%
\end{picture}%
\setlength{\unitlength}{4144sp}%
\begingroup\makeatletter\ifx\SetFigFont\undefined
\def\x#1#2#3#4#5#6#7\relax{\def\x{#1#2#3#4#5#6}}%
\expandafter\x\fmtname xxxxxx\relax \def\y{splain}%
\ifx\x\y   
\gdef\SetFigFont#1#2#3{%
  \ifnum #1<17\tiny\else \ifnum #1<20\small\else
  \ifnum #1<24\normalsize\else \ifnum #1<29\large\else
  \ifnum #1<34\Large\else \ifnum #1<41\LARGE\else
     \huge\fi\fi\fi\fi\fi\fi
  \csname #3\endcsname}%
\else
\gdef\SetFigFont#1#2#3{\begingroup
  \count@#1\relax \ifnum 25<\count@\count@25\fi
  \def\x{\endgroup\@setsize\SetFigFont{#2pt}}%
  \expandafter\x
    \csname \romannumeral\the\count@ pt\expandafter\endcsname
    \csname @\romannumeral\the\count@ pt\endcsname
  \csname #3\endcsname}%
\fi
\fi\endgroup
\begin{picture}(5917,5051)(361,-4583)
\put(4246,-2136){\makebox(0,0)[lb]{\smash{\SetFigFont{10}{12.0}{rm}18}}}
\put(4241,-211){\makebox(0,0)[lb]{\smash{\SetFigFont{10}{12.0}{rm}15}}}
\put(4241,-836){\makebox(0,0)[lb]{\smash{\SetFigFont{10}{12.0}{rm}16}}}
\put(4241,-1586){\makebox(0,0)[lb]{\smash{\SetFigFont{10}{12.0}{rm}17}}}
\put(4241,-2836){\makebox(0,0)[lb]{\smash{\SetFigFont{10}{12.0}{rm}19}}}
\put(5200,-2650){\makebox(0,0)[lb]{\smash{\SetFigFont{10}{12.0}{rm}+}}}
\put(4241,-3486){\makebox(0,0)[lb]{\smash{\SetFigFont{10}{12.0}{rm}20}}}
\put(4241,-4161){\makebox(0,0)[lb]{\smash{\SetFigFont{10}{12.0}{rm}21}}}
\put(2446,-211){\makebox(0,0)[lb]{\smash{\SetFigFont{10}{12.0}{rm}8}}}
\put(2446,-836){\makebox(0,0)[lb]{\smash{\SetFigFont{10}{12.0}{rm}9}}}
\put(2446,-1586){\makebox(0,0)[lb]{\smash{\SetFigFont{10}{12.0}{rm}10}}}
\put(2446,-2136){\makebox(0,0)[lb]{\smash{\SetFigFont{10}{12.0}{rm}11}}}
\put(2446,-2836){\makebox(0,0)[lb]{\smash{\SetFigFont{10}{12.0}{rm}12}}}
\put(2446,-3486){\makebox(0,0)[lb]{\smash{\SetFigFont{10}{12.0}{rm}13}}}
\put(2446,-4161){\makebox(0,0)[lb]{\smash{\SetFigFont{10}{12.0}{rm}14}}}
\put(588,-1586){\makebox(0,0)[lb]{\smash{\SetFigFont{10}{12.0}{rm}3}}}
\put(588,-2136){\makebox(0,0)[lb]{\smash{\SetFigFont{10}{12.0}{rm}4}}}
\put(588,-2836){\makebox(0,0)[lb]{\smash{\SetFigFont{10}{12.0}{rm}5}}}
\put(588,-3486){\makebox(0,0)[lb]{\smash{\SetFigFont{10}{12.0}{rm}6}}}
\put(588,-4161){\makebox(0,0)[lb]{\smash{\SetFigFont{10}{12.0}{rm}7}}}
\put(588,-836){\makebox(0,0)[lb]{\smash{\SetFigFont{10}{12.0}{rm}2}}}
\put(3000,-4861){\makebox(0,0)[lb]{\smash{\SetFigFont{10}{12.0}{rm}Table 4.}}}
\end{picture}

\end{center}

$ $\par

If $\G_\a$ is a disjoint union of a triangle and a $K_4$, every $i \in \rm{core}(F) \setminus V_\a$ is adjacent to either the triangle or adjacent to every triangles of the $K_4$. Let $\b:= \a - \e_t + \e_u + \e_v$ for some $t,u,v$ are vertices of the $K_4$. Then $\b \in \NN^n$ and $\G_{\b}$ is a disjoint union of a triangle and a triangle of weight $(2,2,1)$. Moreover, since for any $i \in \rm{core}(F) \setminus V_\a$ $i$ is adjacent to either the triangle or adjacent to every triangles of the $K_4$,  then $i$ is adjacent to $\G_\b$ and hence $N[V_\a]= N[V_\b]$. It follows that $\nu(\G_\b  - N_\b(i)) \ge 4 - \deg_\b(i)$ for all $i \in \rm{core}(F) \setminus V_\b$. Moreover, due to Theorem \ref{monomial}, $\G_\b$ is $4$-saturating. Therefore, $F$ is a minimal  among the covers containing $N[V_\b]$ with $\G_\b$ is $4$-saturating and $\nu(\G_\b  - N_\b(i)) \ge 4 - \deg_\b(i)$ for all $i \in \rm{core}(F) \setminus V_\b$. This case can be pass to the case $\G_{V_\a}$ is union of two non-adjacent triangles. \par 

If $\G_\a$ is a union of two non-adjacent triangles connected by a path of length 2 then either the endpoints of the path are adjacent in $\G_\a$ or every vertex in $\rm{core}(F) \setminus V_\a$ is adjacent to at least one triangle. If it is in the first situation then $\G_{V_\a}$ is of the type 15 in Table 4. If it is in the last situation then as in the above case, after replacing $\a$ by $\b:=\a-\e_v+\e_i+\e_j$ where $v$ denotes the midpoint of the path and $\{i,j\}$ is an edge of the triangles of $\G_\a$, it can be pass to the case $\G_{V_\a}$ is two non-adjacent triangles.\par

If $\G_\a$ is a triangular prism then the condition on the degree of every vertex of core$(F) \setminus V_\a$ implies that these vertices are adjacent to every pentagon of this triangular prism. Similarly to the above cases, this case can be passed to the case $\G_{V_\a}$ is a pentagon. \par 

If $\G_\a$ is spanned by a cone on a $C_5$ then every vertex of core$(F) \setminus V_\a$ is adjacent to this $C_5$, and this case can be passed to the case $\G_{V_\a}$ is a pentagon. \par 

If $\G_\a$ is spanned by 4 triangles as in the type 3 of Table 2 , then it can be pass to the case $\G_{V_\a}$ is a union of two triangles meeting at a vertex. \par 

If $\G_\a$ is a union of two triangles sharing an edge of weight $(2,2)$ then every vertex of core$(F) \setminus V_\a$ is adjacent to one of these triangles, hence this case can be pass to the case $\G_{V_\a}$ is a triangle. \par 

If $\G_\a$ belongs to one of the other cases, then it is obvious that $\G_{V_\a}$ is one of subgraphs of Table 4.
\end{proof}

\begin{Corollary}  
$\depth R/I^4 > 0$ if and only if $\G$ has no dominating subgraph of the types of Table 4.
\end{Corollary}

\begin{proof} Note that $\depth R/I^4 > 0$ if and only if $\mm = P_V$ is not an associated prime of $I^4.$ Since $V$ is minimal among the covers containing a subset $N$ if and only if $V = N$. Apply Theorem \ref{asso 4} to the case $F = V$, one obtains the assertion.
\end{proof}

\end{document}